\documentclass[reqno,12pt,letterpaper]{amsart}
\usepackage{amsmath,amssymb,amsthm,graphicx,url}
\usepackage{epstopdf}
\usepackage{color}

\newcommand{\CI}{{\mathcal C}^\infty}
\newcommand{\CIc}{{\mathcal C}^\infty_{\rm{c}}}

\newtheorem{theo}{Theorem}
\newtheorem{cor}{Corollary}

\newtheorem{prop}{Proposition}[section]

\numberwithin{equation}{section}
\numberwithin{figure}{section}

\newcommand{\neigh}{\operatorname{neigh}}

\DeclareMathOperator{\supp}{supp}

\title[Weighted eigenfunction estimates]%
{Weighted eigenfunction estimates with applications to compressed sensing}

\author{Nicolas Burq}
\email{nicolas.burq@math.u-psud.fr}
\address{D\'epartement de Math\'ematiques, Universit\'e Paris-Sud,
91405 Orsay Cedex, France} 
\author{Semyon Dyatlov}
\email{dyatlov@math.berkeley.edu}
\address{Department of Mathematics, University of California,
Berkeley, CA 94720, USA}
\author{Rachel Ward}
\email{rward@math.utexas.edu}
\address{Mathematics Department, 
University of Texas at Austin, 
              Austin, TX 78712, USA}
\author{Maciej Zworski}
\email{zworski@math.berkeley.edu}
\address{Department of Mathematics, University of California,
Berkeley, CA 94720, USA}

\begin{document}

\begin{abstract}
Using tools from semiclassical analysis, we give weighted $L^\infty$
estimates for eigenfunctions of strictly convex surfaces of
revolution.  These estimates give rise to new sampling techniques and
provide improved bounds on the number of samples necessary for
recovering sparse eigenfunction expansions on surfaces of revolution.
On the sphere, our estimates imply that any function having an
$s$-sparse expansion in the first $N$ spherical harmonics can be
efficiently recovered from its values at $m \gtrsim s N^{1/6}
\log^4(N)$ sampling points.
\end{abstract}

\maketitle

\section{Introduction}

Consider the sphere and a chosen rotational action 
generated by $ \partial_\varphi $:
\begin{gather*}
{\mathbb S}^2 := \{x \in \mathbb R^3  \, : \, x_1^2+x_2^2+x_3^2=1\}\,,
\ \ \  \partial_\varphi=x_1 \partial_{x_2}-x_2 \partial_{x_1} \,, \\
\mathbb S^2 \ni x = ( \cos \varphi \sin\theta, \sin \varphi \sin
\theta, \cos \theta ) \,, \ \  0 \leq  \theta  \leq \pi \,, \ \ 0 \leq
\varphi \leq 2 \pi \,.
\end{gather*}

Let $k,\ell \in \mathbb Z$, $|k|\leq \ell$, and let $ Y_{\ell}^k
(\varphi,\theta ) $ be the $L^2$ normalized spherical harmonics, the
joint eigenfunctions of the Laplacian in spherical coordinates
$\Delta_{{\mathbb S}^2}$ and the rotational generator
$\partial_{\varphi}$:
\begin{gather}
\label{eq:Ylk}
\begin{gathered}
- \Delta_{{\mathbb S}^2}  Y_\ell^k = 
- \left( \frac{ 1 } { \sin^2 \theta} \partial_\varphi^2 + 
 \frac{ 1 } { \sin \theta} \partial_\theta \left( \sin
    \theta \partial_\theta \right) \right) Y_{\ell}^k = 
l(l+1) Y_\ell^k \,, \ \ \ 
\frac 1 i \partial _\varphi  Y_\ell^k = k Y_\ell^k\,, \\ 
\int_0^{2\pi} \int_0^\pi Y_\ell^k ( \varphi, \theta ) 
 \overline{ Y_{\ell'}^{k'} ( \varphi, \theta ) }\sin \theta   d
  \theta  d \varphi = \delta_{\ell \ell' } \delta_{ k k   '} \,. 
\end{gathered}
\end{gather}

\noindent Applied to the sphere, our main result on weighted $L^{\infty}$ estimates reads
 \begin{theo}
\label{t:sph}
Let $ Y_\ell^k ( \varphi, \theta ) $, $ 0 \leq \varphi \leq 2 \pi$, $ 0
\leq \theta\leq \pi $, be the spherical harmonics defined above. Then
for $ \ell \geq 1 $, 
\begin{equation}
\label{eq:ts1}
|\sin^2 \theta  \cos \theta |^{1/6}  \, |  Y_\ell^k (
\varphi, \theta ) | \leq C \ell^{1/6} \,,
\end{equation}
where $ C $ is a universal constant.
\end{theo}  

The power $ \ell^{1/6 } $ in \eqref{eq:ts1} can be explained as
follows. Taking the Fourier expansion in $ \varphi$ reduces the first
differential equation in~\eqref{eq:Ylk} to
\[ \left(  \frac{1 } { \sin \theta} h \partial_\theta \left( \sin
    \theta h \partial_\theta \right)  + \frac{ \alpha^2 } { \sin^2
    \theta} - 1 \right)  u  =  0,  \ \ 
h = (\ell ( \ell + 1 ) )^{-1/2} , \ \ \alpha = h^2 k^2 , \] where $ k
\in {\mathbb Z} $ is an eigenvalue of $ \frac 1 i \partial_\varphi $
and $Y_\ell^k=u(\theta)e^{ik\varphi}$.  When $\alpha$ is such that $ 0
< \epsilon < |\alpha | < 1 -\epsilon $, this equation has two turning
points at $ \sin \theta = \pm \alpha$.  Physically this corresponds to
caustic formation: the focusing at turning points increases the
intensity of the wave function, that is, it increases its $ L^\infty $
norm by a factor of $ h^{-1/6} \sim \ell^{1/6} $~-- see Proposition
\ref{l:well-eig} below.%
\footnote{The $h^{-1/6}$ factor can be seen on the model example
of the equation $(h^2 D_x^2+x)v=0$, with a turning point at $x=0$ and
locally $L^2$-normalized solution $h^{-1/6}Ai(h^{-2/3} x)$, where $Ai$
is the Airy function. The $h^{-1/6}$ normalization here follows from
the asymptotic behavior of $Ai(y)$ as $y\to -\infty$.}
Since the $h^{-1/6}$ loss happens all over the sphere, such growth in
the $ L^\infty $ norm cannot be eliminated by a weight function. In
order to get a uniform bound on the entire sphere in \eqref{eq:ts1},
we choose a weight function vanishing at the pole and the equator.  A
more detailed explanation of the weights and the principles of
semiclassical analysis on which the analysis is based is given at the
end of Section~\ref{wee}.

\subsection{Motivation}

Consider functions on the sphere which are \emph{bandlimited} and \emph{sparse}:
\begin{equation}
\label{spheresparse}
f(\varphi,\theta) = \sum_{\ell =0}^{\sqrt{N}-1} \sum_{k=-\ell}^{\ell} c_{\ell,k} Y_{\ell}^k (\varphi, \theta); \quad \textrm{at most $s < N$ of the $c_{\ell,k}$ are nonzero.}
\end{equation}
Functions well-approximated as bandlimited and sparse arise in
applications ranging from models for protein structure \cite{protein}
to cosmic microwave background (CMB) data \cite{cmb}.  In \cite{RaW},
Rauhut and Ward showed that such functions can be efficiently
reconstructed from far less information than their ambient dimension
suggests; in particular, they show that for certain sets of sampling
points $(\varphi_i, \theta_i) \in \mathbb{S}^2, \hspace{1mm} j \in
[m],$ of size
\begin{equation}
\label{raw:m}
m \gtrsim s N^{1/4} \log^{4} N,
\end{equation}
any function $f$ of the form \eqref{spheresparse} can be reconstructed
from its values $f(\varphi_i, \theta_i)$ as the function of this
bandwidth whose coefficient vector $c = c_{\ell, k}$ has minimal
$\ell_1$-norm $\| c \|_1 = \sum_{\ell=0}^{\sqrt{N}-1}
\sum_{k=-\ell}^{\ell} | c_{\ell,k} |$.  It is shown that $m$ angular
coordinates $(\theta_i, \varphi_i)$ where $m$ satisfies \eqref{raw:m},
drawn independently from the measure $d\theta d\varphi$ on $[0,\pi]
\times [0,2\pi]$, will almost always be a set of sampling points for
which this holds.

In Section $2$ we show how Theorem \ref{t:sph} improves on the results
in \cite{RaW}, strengthening the required number of sampling points
for recovering functions of the form \eqref{spheresparse} to
\begin{equation}
\label{bdwz:m}
m \gtrsim s N^{1/6} \log^4(N)
\end{equation}
by drawing angular coordinates $(\theta_i, \varphi_i)$ independently
from the measure $|\tan(\theta) |^{1/3} d\theta d\varphi$ on $[0,\pi]
\times [0,2\pi]$.  The specific statement is given in Corollary
\ref{cor3}.  As seen in Figure~\ref{fig:1a}(c), this measure generates
higher sampling density around the poles and equator; the measure
$d\theta d\varphi$, illustrated in Figure \ref{fig:1a}(b) and on which
the analysis of \cite{RaW} is based, only generates higher sampling
density at the poles.

It remains open whether there exists a sampling strategy for which the
factor of $N^{1/6}$ in \eqref{bdwz:m} can be provably eliminated. As
the discussion following Theorem 1 indicates, such a result cannot be
done by using weight functions alone.

\subsection{Sparse recovery for arbitrary surfaces of revolution.}

The weighted $L^{\infty}$ estimates given in Corollary~\ref{cor2} of
Section~\ref{wee} provide sampling strategies more broadly for
recovering sparse eigenfunction expansions on any strictly convex
surface of revolution.  In particular, assume that $M$ is a strictly
convex surface of revolution parametrized by $(r,\varphi)\in
[r_-,r_+]\times [0,2\pi)$.  The induced Riemannian metric on $ M $ is
given by
\[   g = dr^2 + a ( r)^2 d \varphi^2 ,  \ \ a ( r ) = ( r_+ - r) ( r -
r_-) b ( r) , \ \ b ( r_\pm) > 0 ,  \ \ b' ( r_\pm) = 0 ,  \]
where $ a ( r ) $ has a unique nondegerate local maximum at $ r= r_0
$, $ r_- < r_0 < r_+ $: $ a' ( r ) \neq 0 $, $ r \neq r_0 $, $ a'' (
r_0 ) < 0 $.  In particular, using Corollary~\ref{cor2} in
Section~\ref{wee}, we will prove the following.
\begin{prop}
\label{theo:sphere}
Suppose that $M$ is a strictly convex surface of revolution and
consider $\psi_j$, the ($L^2$-normalized) joint eigenfunctions of the
Laplace-Beltrami operator on $M$ and the rotational generator $
\frac{1}{ i}\partial_\varphi$. 

\noindent Let $m, s$, and $N$ be given integers satisfying
 \begin{equation}
 \label{eq:msphere}
m \gtrsim s N^{1/6}\log^4(N),
 \end{equation}
and suppose that $m$ coordinates $(\varphi_i, r_i)$ are drawn
independently according to the measure
$$\left( \frac{a(r)}{|r - r_0|} \right)^{1/3} dr  d \varphi$$
on $[r_{-}, r_{+}] \times [0,2\pi)$. Consider the associated $m \times
N$ sampling matrix $A$ with entries
$$
A_{i,j} = a(r_i)^{-2/3} | r_i - r_0 |^{-1/3} \psi_j (\varphi_{i}, r_{i}).
$$
With probability exceeding $1-N^{-\log^3(s)}$ the following holds for
all $s$-sparse functions
\[ 
f(\varphi, r) = \sum_{j=1}^{N} c_{j} \psi_j(\varphi, r), \quad \quad | c | \leq s:\]
Suppose that sample values $y_i =f(\varphi_i, r_i)$ are known, and let
\begin{equation}
\label{ell1sphere'}
{c}^{\sharp} = \arg \min \| z \|_1 \hspace{3mm} \textrm{subject to} \hspace{3mm}  A z = y.
\end{equation}
Then $c = c^{\#}.$  That is, $f$ is recovered exactly via \eqref{ell1sphere'}.
\end{prop}
\noindent Applying Proposition~\ref{theo:sphere} to the sphere, we get in particular
\begin{cor}
\label{cor3}
Given $m \gtrsim s N^{1/6} \log^{4} N$ sampling points on the sphere
with angular coordinates $(\theta_i, \varphi_i)$ drawn independently
from the measure $| \tan^{1/3}(\theta) | d\theta d\varphi$ on $[0,\pi]
\times [0,2\pi]$ , with high probability any $s$-sparse function of
the form \eqref{spheresparse} can be recovered exactly as the
minimizing function of the convex program \eqref{ell1sphere'} with
$A_{i,j} = | \sin^2{\theta_i}\cos{\theta_i} |^{1/3} \psi_j(\varphi_i,
\theta_i)$.
\end{cor}

\begin{figure}
{\includegraphics[width=5cm, height=5cm]{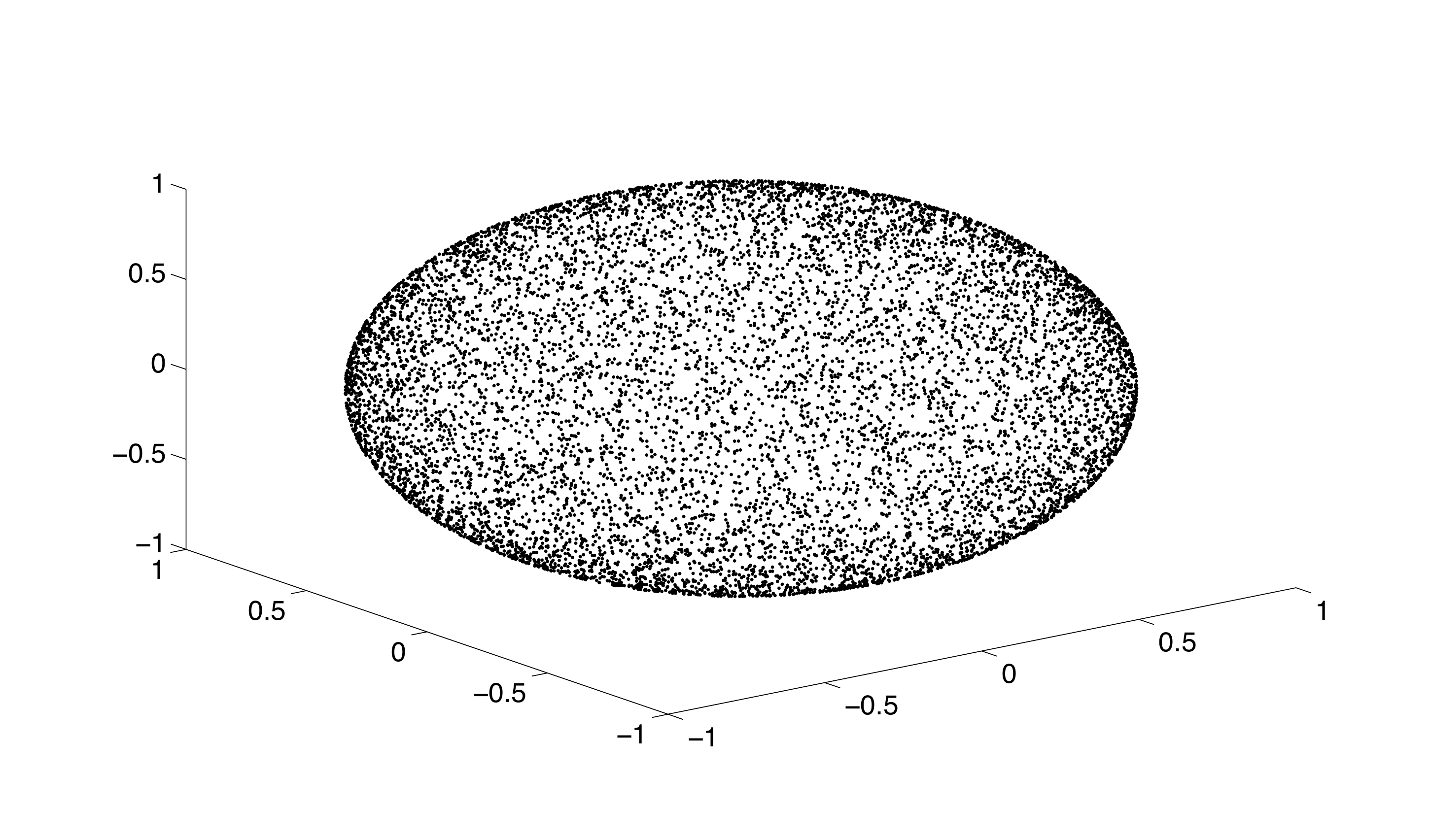}}
{\includegraphics[width=5cm, height=5cm]{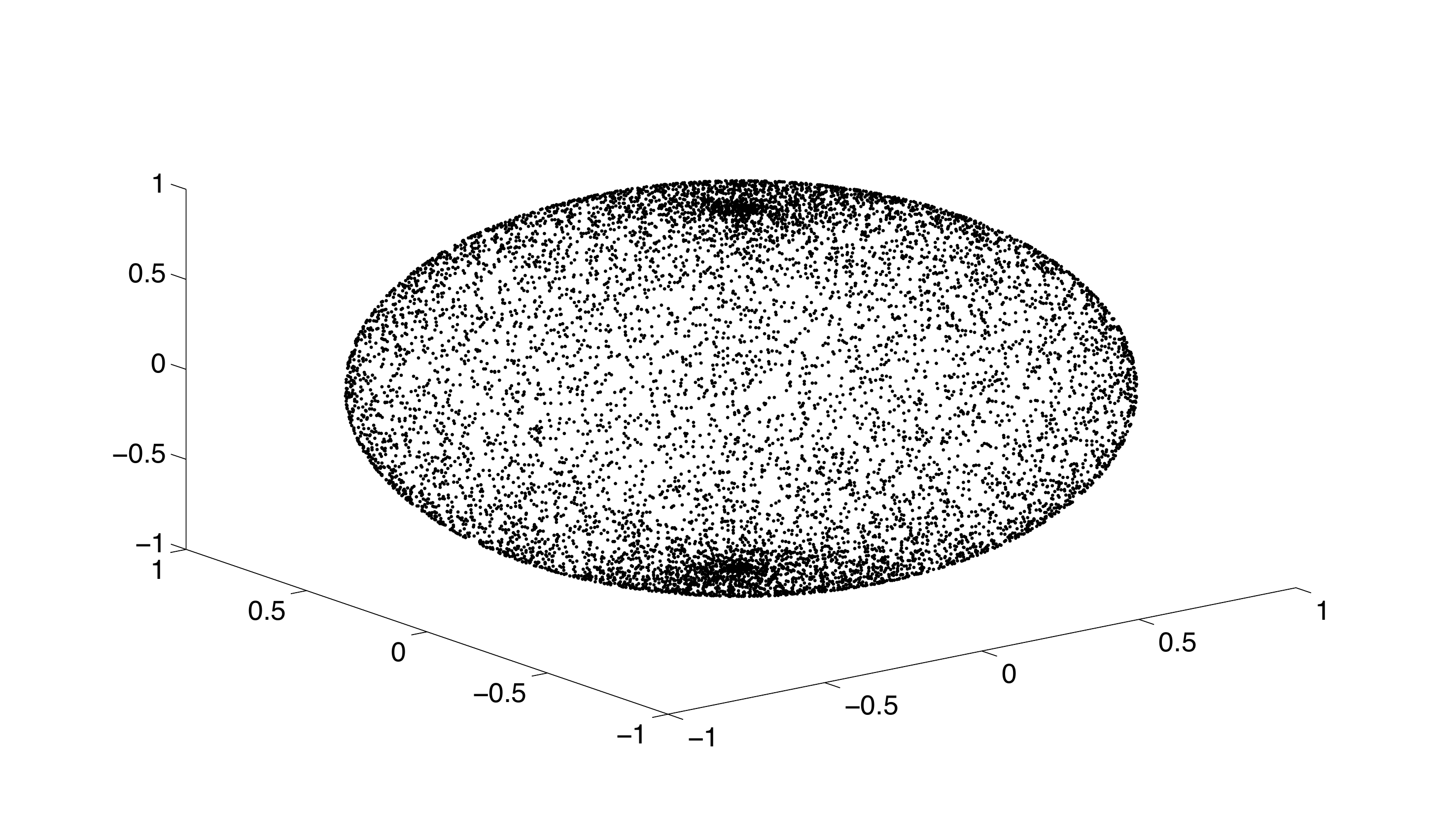}}
{\includegraphics[width=5cm, height=5cm]{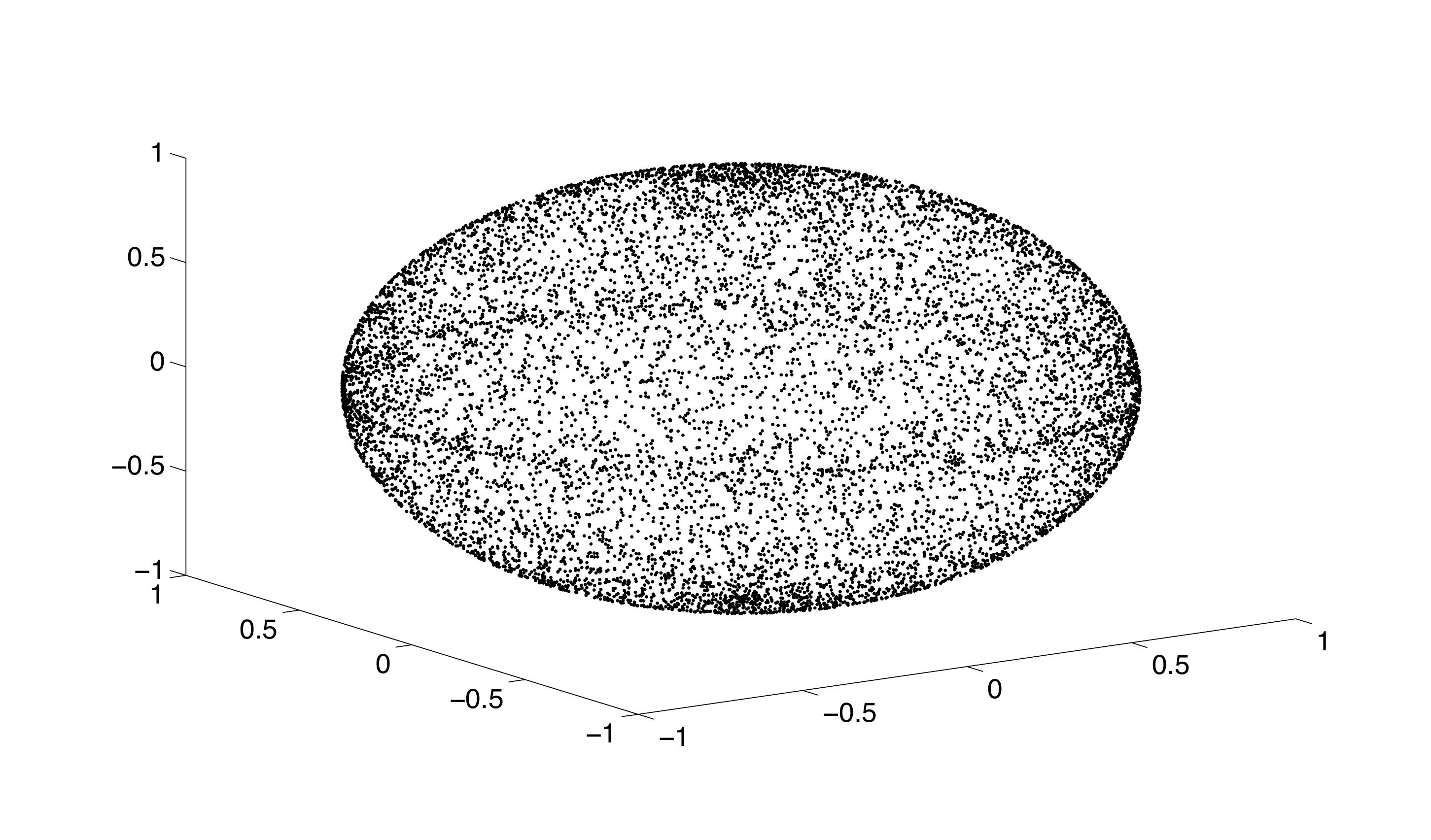}}
\hbox to\hsize{\hskip1in (a) \hss (b) \hss (c) \hskip1in}
\caption{$m=10,000$ independent  draws from the spherical measures 
(a) $\sin(\theta)\,d\theta  d\varphi$, (b) $d\theta 
d\varphi$, and (c) $| \tan(\theta) |^{1/3}\, d\theta  d\varphi$.
}\label{fig:1a}
\end{figure}

\subsection{Numerical experiments}
In this section we test the numerical relevance of Corollary
\ref{cor3}, comparing the rate of correct reconstruction of sparse
bandlimited functions on the sphere \eqref{spheresparse} via the
$\ell_1$-minimizer \eqref{ell1sphere'} when $m$ sampling points
$(\theta_i, \varphi_i)$ are drawn i.i.d. from the measures (a)
$\sin(\theta)\,d\theta d\varphi$, (b) $d\theta d\varphi$, and (c) $|
\tan(\theta) |^{1/3}\, d\theta d\varphi$.  More specifically, for each
choice of sampling measure, we vary a number of sampling points $m$
between $1$ and $N$, and vary a sparsity level $s$ between $1$ and
$m$.  For each choice of $m$ and $s$, we generate $50$ $s$-sparse
bandlimited functions by repeatedly choosing a support of $[N]$ of
size $s$ at random, and prescribing to the chosen support
i.i.d. Gaussian coefficients.

From left to right, the phase diagrams in Figure \ref{fig:1}
correspond to sampling measures (a) $\sin(\theta) d\theta d\varphi$,
(b) $d\theta d\varphi$, and (c) $| \tan(\theta) |^{1/3} d\theta
d\varphi$.  White indicates complete recovery, and black indicates no
recovery whatsoever.  It is clear that the sampling strategies $(b)$
or $(c)$ give better results than $(a)$.  Diagrams (b) and (c) both
show a sharp transition between complete recovery and no recovery
whatsoever as the ratio $s/m$ increases as a function of $m/N$.
However, the region of phase space corresponding to complete recovery
is noticeably larger in (3) when $m/N$ is large.  Note that when
$m=N$, all three sampling schemes should give perfect reconstruction
as the system of equations $y = A c$ in the minimization problem
\eqref{ell1sphere'} has a unique solution with probability $1$.
However, plots (a) and (b) show zero reconstruction around this point,
an artifact of round-off error due to the ill-conditioning of the
sampling matrix $A$.

\begin{figure}

{\includegraphics[width=5cm, height=5cm]{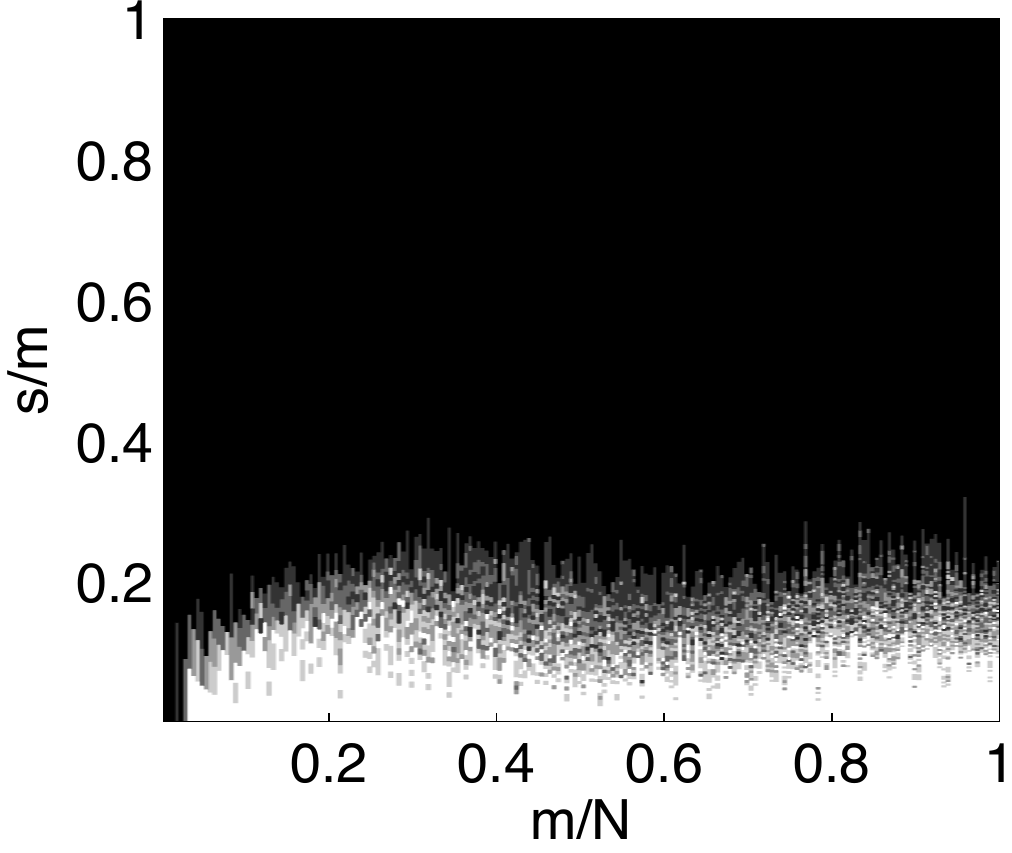}}
\includegraphics[width=5cm, height=5cm]{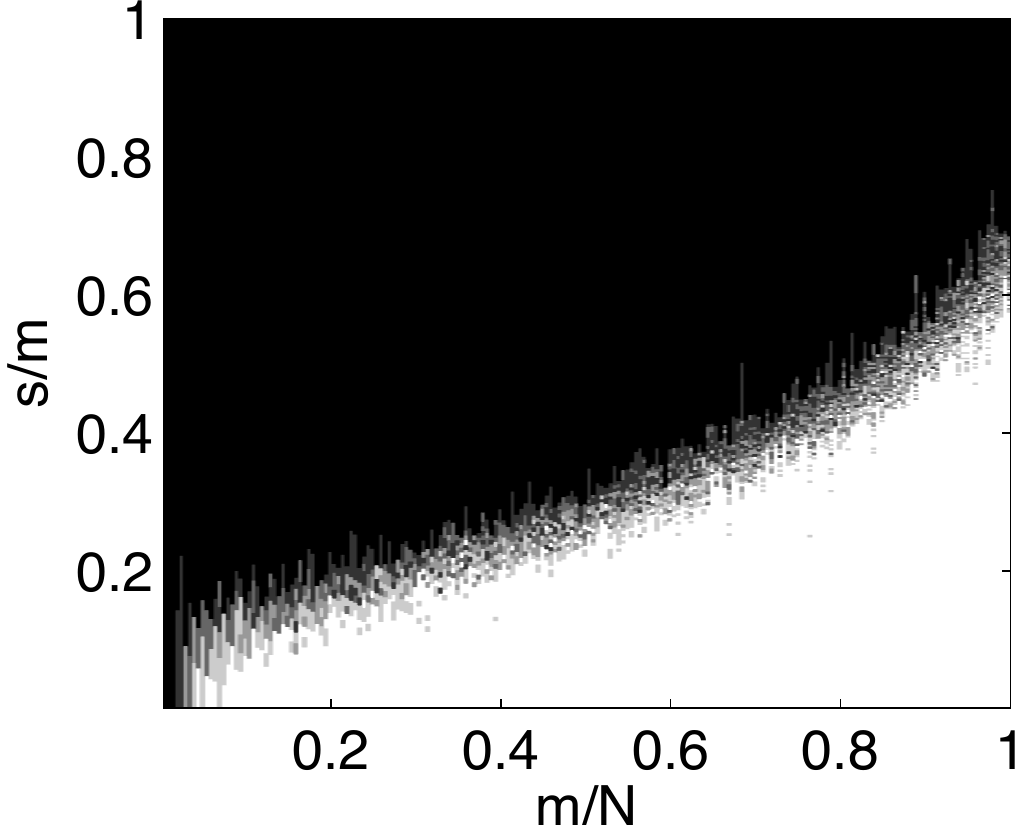}
\includegraphics[width=5cm, height=5cm]{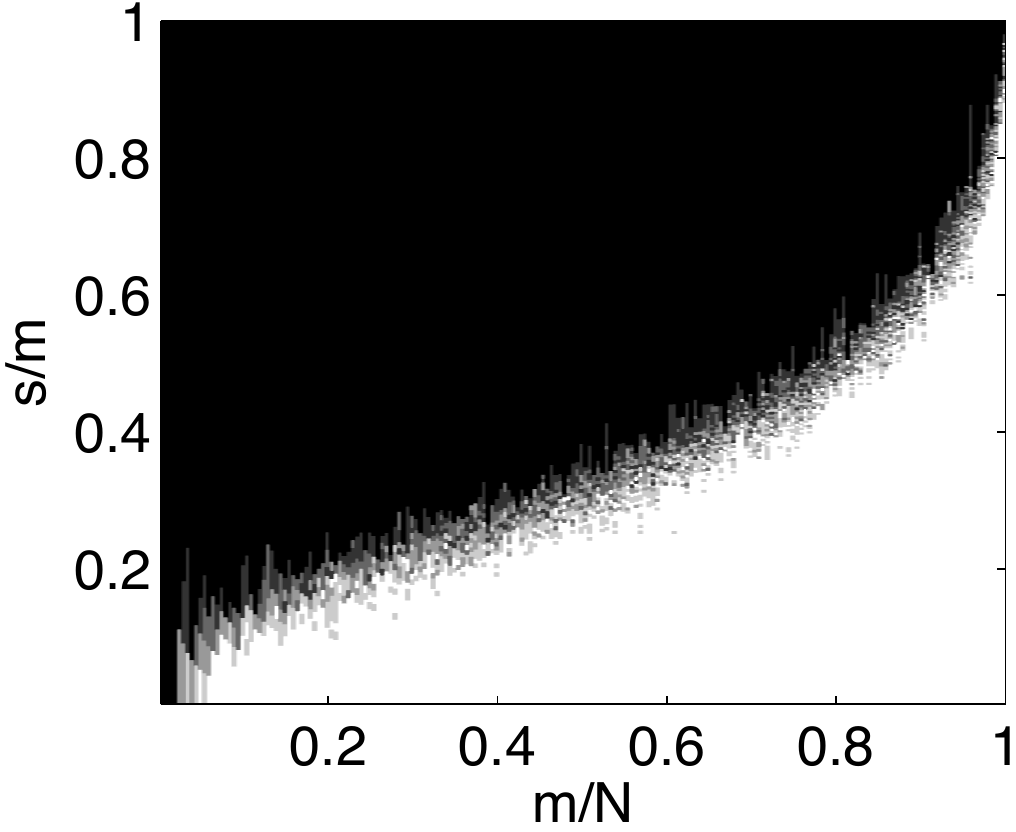}
\hbox to\hsize{\hskip 1.2in (a) \hss (b) \hss (c) \hskip.9in}
\caption{Phase diagrams illustrating transition between uniform
  recovery (white) to no recovery whatsoever (black)
  of spherical harmonic expansions $f(\varphi, \theta) =
  \sum_{\ell=0}^{19} \sum_{k=-\ell}^{\ell} c_{\ell, k}
  Y_{\ell}^k(\varphi, \theta)$ of sparsity level $| c | \leq s$ from $m$ samples $f(\varphi_i, \theta_i)$.  In (a),
 sampling points are drawn from the
  volume measure $\sin\theta\,d\theta d\varphi$.  In (b) the $m$ sampling points are drawn from
  $d\theta d \varphi$, and in (c) the sampling points are drawn from $| \tan(\theta)|^{1/3}\, d\theta
  d\varphi$.}
  \label{fig:1}
\end{figure}

\noindent {\bf Organization of the paper.}
In Section \ref{app} we review the relationship between sparse
recovery techniques on manifolds and weighted $ L^\infty $ bounds on
the associated eigenfunctions.  We then show how the main results of
this paper strengthen and generalize existing sparse recovery bounds.
The generalization of Theorem \ref{t:sph} to arbitrary convex surfaces
of revolution is given in Theorem \ref{t:1} of Section \ref{wee},
while Section \ref{prel} provides a detailed account of preliminaries
from semiclassical analysis needed for the proof which is presented in
Section~\ref{proo}.

\medskip

\noindent {\bf Notation.} 
In the paper $ C $ denotes a constant, independent of asymptotic
parameters, but changing depending on the context.  We use the usual $
{\mathcal O} $ notation with subscripts to indicate that the
associated constant might depend on the variable in the subscript, for
instance $ f = {\mathcal O}_x ( g ) $ means that $ f ( x, y ) \leq C (
x) g ( y ) $ for some $ C ( x ) $ depending on $ x $. We follow the
basic notational convention listed in \cite[Appendix A]{e-z}.
Consequently the above notation should not be confused with $ u =
{\mathcal O}_V ( g) $ for $ V$ a Hilbert space; the latter means that
$ \| u \|_V \leq C g $.  The notation $ f \lesssim g $ means that
there exists $ C$ such that $ f \leq C g $.  Finally, we use the
shorthand $[N] = \{1,2, ... ,N\}$. For a vector $x \in \mathbb{C}^N$
or $x \in \mathbb{R}^N$, we indicate the size of the support of by $|
x | = \{ \# j: | x_j | > 0 \}$.

\section{Compressed sensing and weighted $L^{\infty}$ estimates}
\label{app}
Suppose we have a finite system of functions $\{ \psi_j, j \in [N] \}$
on a compact manifold $M$.  Suppose we also have a function $f: {M}
\rightarrow \mathbb{C}$ which is $s$-sparse with respect to this
function system,
\begin{equation}
\label{f:sparse}
f  = \sum_{j=1}^N c_j \psi_j, \ \ \ \ | c | \leq s < N.
 \end{equation}
The area of compressed sensing \cite{crt} is concerned with the
following questions.  For a given system $\{ \psi_j \}$ and $s$-sparse
function $f$ of the form \eqref{f:sparse}, how many samples $f(x_i)$
where $x_i \in M$ do we need to uniquely identify $f$? Is it possible
moreover to efficiently and robustly reconstruct such a function from
these samples?  That is, to distinguish an arbitrary linear
combination of $N$ known functions $\psi_j$ we would clearly need $N$
samples.  But if we know a priori that $f$ is $s$-sparse, and if the
locations of the $s$ nonzero coefficients $c_j$ are known, then we
would need only $s$ samples.  When the locations of the $s$
coefficients are not known, $2s$ samples still suffice in certain
situations.  Namely, consider the matrix $\Psi \in
\mathbb{C}^{m \times N}$ with entries $\Psi_{i,j} = \psi_j(x_i)$, and
observe that
\begin{equation}
\label{yfromf}
y = \big( f(x_1), f(x_2), \dots, f(x_m) \big)^t = \Psi c.
\end{equation}
Each $s$-sparse function $f$ has a distinct image $y = \Psi c$ if
every sub-matrix of the $m \times N$ matrix $\Psi$ consisting of at
most $2s$ columns is non-singular, and this is true for many $m \times
N$ matrices having only $m = 2s$ rows (consider matrices having
i.i.d. Gaussian entries, for example.)  Subject to this condition, one
could solve for the unique $s$-sparse solution to $y = \Psi c$ by
searching over all $s$-sparse vectors $c$.  However, in general this
is an NP-hard problem.  As it turns out, polynomial-time recovery of
sparse solutions is possible if all $2s$-column sub-matrices of $\Psi$
are not only nonsingular, but well-conditioned - a property that can
only hold if $\Psi$ has at least $m \gtrsim s \log(N)$ rows
\cite{crt}. In the compressed sensing literature, a matrix $\Psi \in
\mathbb{C}^{m \times N}$ is said to have the \emph{restricted isometry
property} of order $2s$ if, for a fixed parameter $\delta < 1$,
\begin{equation}
\label{rip}
(1-\delta) \| u \|_2 \leq \| \Psi u \|_2 \leq (1 + \delta) \| u \|_2, \quad \forall u: | u | \leq 2s.
\end{equation}
As shown in \cite{crt}, if a matrix $\Psi \in \mathbb{C}^{m \times N}$
has this property, and if $y = \Psi c$ for some $s$-sparse vector $c
\in \mathbb{C}^N$, then $c$ is guaranteed to also be the vector of
minimal $\ell_1$-norm among solutions $c'$ to the underdetermined
system $\Psi c' = y$.  As $\ell_1$-minimization can be solved
efficiently using linear programming, the sparse coefficient vector
$c$ can be reconstructed efficiently.  Moreover, given any arbitrary
vector $c \in \mathbb{C}^N$ with \emph{best $s$-sparse approximation
error}
\begin{equation}
\label{bestapprox}
\varepsilon = \min_{z \in \mathbb{C}^N:  |z| \leq s} \| c - z \|_1,
\end{equation}
and the minimizing solution
$$c^{\sharp} = \arg \min \| z \|_1 \quad \textrm{subject to }  \| \Psi z  - \Psi c \|_2 \leq \epsilon,$$
then $\| c - c^{\sharp} \|_2 \lesssim \varepsilon/\sqrt{s}$.

\subsection{Sparse recovery for bounded orthonormal systems.}
In general it is hard to verify the restricted isometry property
\eqref{rip} holds for a given matrix $\Psi$, but in the following
set-up it can be assured with high probability.  Suppose we have a
system of functions $\{ \psi_j$, $j \in [N] \}$ which are orthonormal
on a measurable space ${M}$ endowed with a probability measure $\nu$,
i.e.
 \begin{equation}
\int_{M} \psi_j(x) \overline{\psi_i(x)} d\nu(x) = \delta_{i,j}, \quad i,j \in [N].
\end{equation} 
Suppose further that $m$ sampling points $x_i \in M$ are drawn
independently from the orthogonalization measure $\nu$.  Then, as
shown in \cite{strucrand}, with high probability with respect to the
draw of the sampling points, the normalized sampling matrix
$\frac{1}{\sqrt{m}}\Psi $, where $\Psi_{i,j} = \psi_j(x_i)$, satisfies
\eqref{rip} as long as the number of samples $m \gtrsim B^2 s
\log^4(N)$, where
\begin{equation}
\label{ubounded}
B = \max_j \| \psi_j \|_{\infty}.
\end{equation}
The parameter $B$ should be interprested as a measure of
\emph{incoherence} between the basis $\psi_j$ and pointwise
measurements; the smaller $B$, the fewer number $m$ of sampling points
$\big( f(x_1), f(x_2), ..., f(x_m)\big)^{t} = \Psi c$ are needed to
recover sparse expansions \eqref{f:sparse}.  This can be interpreted
as a discrete Heisenberg uncertainty principle \cite{dostarck}.  A
precise statement follows.

\begin{prop}
\label{thm:BOS} 
Suppose $ \{ x_i: i \in [m] \}$ is a set of independent and
identically distributed (i.i.d.) sampling points drawn from the
orthogonalization measure $\nu$ associated to an orthonormal system of
functions $\{ \psi_j, j \in[N] \}$ with uniform bound $B = \max_j \|
\psi_j \|_{\infty}$.  If
\begin{equation}\label{BOS:RIP:cond}
m \gtrsim B^2 s \log^4(N),
\end{equation}
then with probability at least $1-N^{-\log^3(s)},$ the following holds
for all $f(x) = \sum_{j=1}^{N} c_j \psi_j(x)$ with $s$-term
approximation error $\varepsilon$ as in \eqref{bestapprox}.  Given $m$
observations $y_{i} = f(x_{i})$, or more concisely $y = \Psi c$, and
the minimizer
\begin{equation}\label{l1eps:prog}
c^\# = \arg \min_{z \in \mathbb{C}^N}  \| z \|_1 \mbox{ subject to }  \frac{1}{\sqrt{m}} \| \Psi z  - \Psi c \|_2 \leq \epsilon,
\end{equation}
it follows that
\begin{equation}
\label{l12noise}
\|c - c^\#\|_2 \lesssim \varepsilon/\sqrt{s}.
\end{equation}
In particular, if $f$ is $s$-sparse then reconstruction is exact, $c^\# = c$.
\end{prop}

Let us apply Proposition \ref{thm:BOS} to a concrete example.  The
orthonormal system of complex exponentials $\psi_j (t) = e^{2\pi i j
t}, t \in [0,1],$ has optimal uniform bound $B= \max_{j} \| \psi_j
\|_{\infty} = 1$.  Applying Proposition \ref{thm:BOS}, we see that $m
\gtrsim s \log^4(N)$ sampling points drawn independently from the
uniform measure on $[0,1]$ will be sufficient to identify any
$s$-sparse trigonometric polynomial of degree at most $N$.

The complex exponentials are eigenfunctions of the Laplacian on the
circle. More generally, for a compact $ n$-dimensional Riemannian
manifold, the $L^2$-normalized eigenfunctions with eigenvalue
$\lambda$ are bounded in $L^{\infty}$ by $\lambda^{(n-1)/4}$ -- see
\cite[Section 7.4]{e-z} and references given there.  Since the number
of eigenvalues less than $ \lambda $ behaves like $ N = \lambda^{n/2}
$ (see \cite{Horm} or \cite[Section 14.3]{e-z}), we obtain a uniform
bound on the first $N$ eigenfunctions of a general $n$-dimensional
manifold:
\[    B \simeq N^{ \frac{ n-1} { 2 n } } . \]

Applying Proposition~\ref{thm:BOS} immediately gives the following.
\begin{cor}
\label{cor1}
Let $ ( M , g ) $ be a compact $ n$-dimensional Riemannian manifold
and let $\{ \psi_j, j \in [N] \} $ be the first $ N$ eigenfunctions of
the Laplacian on $ M $ (with respect to the ordering of
eigenfunctions).

Suppose $ \{ x_i : i \in [m] \}$ is a set of independent and
identically distributed sampling points drawn from the measure given
by the Riemannian volume.  If the number of sampling points satisfies
\begin{equation}\label{BOS:man1}
m \gtrsim  s  N^{\frac{ n-1} n } \log^4(N),
\end{equation}
then with probability at least 
$1-N^{-\log^3(s)},$ the following  holds for all $f(x) = \sum_{j=1}^{N} c_j \psi_j(x)$.

If $\Psi$ is the sampling matrix associated to the eigenfunctions and sampling points, and if $ c^\# $ is defined as in \eqref{l1eps:prog} then 
\[ 
\|c - c^\#\|_1 \lesssim \varepsilon,
\]
where $\varepsilon$ is the $s$-term approximation error \eqref{bestapprox}.
\end{cor}
As $n \rightarrow \infty$, the bound \eqref{BOS:man1} becomes weaker
and weaker.  However, even when $n=1$, Proposition \ref{thm:BOS} can
be rather restrictive for certain function systems.  Consider the
($L^2$-normalized) Legendre polynomials $\{ P_j\}$, the unique
orthonormal polynomials with respect to the Lebesgue measure on
$[-1,1]$.  The Legendre polynomials satisfy $B = \| P_j \|_{\infty} =
|P_j(1)| = (j+1)^{1/2}$.  In this situation, Proposition \ref{thm:BOS}
gives only that $m \gtrsim sN\log^4(N)$ measurements are necessary for
identifying functions with an $s$-sparse expansion in the first $N$
Legendre polynomials - a trivial estimate.  Proposition \ref{thm:BOS}
can however be adapted to give meaningful estimates in a more general
setting, as introduced in \cite{RW2}.

\begin{prop}
\label{thm:BOS2}
Let $\{ \psi_j , j \in [N] \} $ be an orthonormal system of functions
on a probability space ${{M}}$ with orthogonalization measure $\nu$.

Suppose that $\omega: {{M}} \rightarrow \mathbb{R}$ satisfies
$\int_{{M}} \omega(x) \nu(x) dx = 1$, and suppose that the functions
$Q_j(x) = \omega(x)^{-1/2} \psi_j(x)$ are bounded:
\begin{equation}
\label{eq:QkK}
\sup_{j \in [N]} \sup_{x \in {M}} |Q_j(x)| \leq K.
\end{equation}
Suppose $ \{ x_i: i \in [m] \}$ are i.i.d. sampling points from the
\emph{composite} orthogonalization measure $\mu = \omega \nu$, and let
$A$ be the preconditioned sampling matrix with entries $A_{i,j} =
Q_j(x_i) = \omega(x_i)^{-1/2} \psi_j(x_i)$.

If $m \gtrsim K^2 s \log^4(N)$ then with probability at least
$1-N^{-\log^3(s)}$ the following holds for all $f(x) = \sum_{j=1}^N
c_{j} \psi_{j}(x)$ with best $s$-term approximation error
$\varepsilon$.  Given $m$ observations $y_{k} = f(x_{k})$, or more
concisely $y = A c$, and the minimizer
\begin{equation}\label{l1eps:prog'}
c^\# = \arg \min_{z \in \mathbb{C}^N}  \| z \|_1 \mbox{ subject to }  \frac{1}{\sqrt{m}} \| A z  -A c \|_2 \leq \epsilon,
\end{equation}
it follows that
\begin{equation}
\label{l12noise'}
\|c - c^\#\|_2 \lesssim \varepsilon/\sqrt{s}.
\end{equation}\end{prop}

\begin{proof}
Apply Proposition \ref{thm:BOS} to the system $\{ Q_{j} \}_{j \in
[N]}$, which is a bounded orthonormal system: $\sup_{j \in [N]} \| Q_j
\|_{\infty} \leq K$ and the $Q_j$ are orthonormal with respect to the
composite measure $\mu = \omega \nu$.
\end{proof}

Proposition \ref{thm:BOS2} quantifies the link between weighted
$L^{\infty}$ estimates on orthonormal function systems and sparse
recovery guarantees. For the Legendre polynomials, which satisfy the
weighted $L^{\infty}$ estimate $$(1-x^2)^{1/2} | P_j(x) | \leq 2
\sqrt{\pi},$$ Proposition \ref{thm:BOS2} gives that $m \gtrsim
s\log^4{N}$ sampling points from the \emph{Chebyshev} measure $\mu(x)
= \sin^{-1/2}(x)dx$ are sufficient for recovering $s$-sparse
expansions in the first $N$ Legendre polynomials.  For details, see
\cite{RW2}.

Below we summarize how the weighted $L^{\infty}$ estimates of this
paper improve and generalize previous results and give rise to
Proposition~\ref{theo:sphere}.

\begin{enumerate}
\item  In \cite{krasikov}, Krasikov proves the following weighted $L^{\infty}$ estimate for the spherical harmonics $Y_{\ell}^k$:
 $$
(\sin \theta) ^{1/2} | Y_{\ell}^k ( \varphi, \theta ) | \lesssim
\ell^{1/4}.
$$
This estimate implies 
\begin{equation}
\label{krasbound'} \sup_{0 \leq \ell \leq \sqrt{N}-1} \sup_{-\ell \leq k \leq \ell} ( \sin \theta) ^{1/2}  | Y_{\ell}^k ( \varphi, \theta ) | \lesssim N^{1/8}.
\end{equation}
In \cite{RaW}, the authors apply Proposition \ref{thm:BOS2} using this
estimate to conclude that $m \gtrsim s N^{1/4} \log^{4} N$ sampling
points on the sphere with angular coordinates $(\theta_i, \varphi_i)$
drawn independently from the measure $d\theta d\varphi$ on $[0,\pi]
\times [0,2\pi)$ suffice for recovering bandlimited sparse spherical
harmonic expansions of the form \eqref{spheresparse}.  This improved
on the $m \gtrsim s N^{1/2} \log^4(N)$ required sampling points given
by Corollary \ref{cor1}, if sampling points are drawn i.i.d. from the
uniform distribution on the sphere.

\item The weighted $L^{\infty}$ estimate given in Theorem \ref{t:sph}
provides even stronger sparse recovery guarantees for spherical
harmonic expansions.  Corollary \ref{cor3} results from applying this
estimate to Proposition \ref{thm:BOS2}: $m \gtrsim s N^{1/6}
\log^4(N)$ sampling points $(\theta_i, \varphi_i)$ from the measure $|
\tan(\theta) |^{1/3} d\theta  d\varphi$ suffice for recovering sparse bandlimited spherical harmonic expansions \eqref{spheresparse}. 
\item 
As summarized in Proposition \ref{theo:sphere}, the weighted
$L^{\infty}$ estimate of Corollary \ref{cor2} gives rise to more
general sampling strategies for recovering sparse eigenfunction
expansions on strictly convex surfaces of revolution.
\end{enumerate}

\section{Weighted eigenfunction estimates for surfaces of revolution}
\label{wee}

If $M$ be a smooth surface of revolution, let $\partial_\varphi$ be
the vector field generating the action of the circle $\mathbb
S^1=\mathbb R/(2\pi \mathbb Z)$ on $M$ by rotations around the axis of
revolution. Denote by $\Delta$ the Laplace--Beltrami operator on $M$,
and by $ D_\varphi $ the self-adjoint operator $ \frac 1
i \partial_\varphi $ which commutes with $ \Delta$. This follows the
standard convention for the operators quantizing momenta.

Let $h>0$ be a small parameter, and assume that $u\in \CI(M)$
satisfies the conditions
\begin{gather}
\label{e:original-cond-1}
\|u\|_{L^2}\leq C_0 \,,\\
\label{e:original-cond-2}
\|(-h^2\Delta-1)u\|_{L^2}\leq C_0h\,,\\
\label{e:original-cond-3}
(hD_\varphi-\alpha)u=0 \,. 
\end{gather}
Here $\alpha\in h \mathbb Z$ varies in a fixed compact set and $C_0$
is some fixed constant. Both $-h^2\Delta-1$ and $hD_\varphi-\alpha$
are semiclassical differential operators; we will freely use the
notation of semiclassical analysis that can be found, for example,
in~\cite[Chapter 4]{e-z}.

We also assume that $u$ satisfies the following
localization assumption: there exists a compactly microlocalized
operator $X(h)$ (that is, $X(h)=\psi^w(x,hD_x)+\mathcal O_{H^{-N}_h\to
H^N_h}(h^N)$ for some $\psi\in \CIc (T^*M)$ and each $N$) and
fixed constants $C_N$ such that for each $N$,
\begin{equation}\label{e:original-cond-4}
\|(1- X(h) )u\|_{H^N_h}\leq C_Nh^N. 
\end{equation}

\vspace{5mm}

{\bf Remark:} 
Conditions~\eqref{e:original-cond-1}, \eqref{e:original-cond-2},
and~\eqref{e:original-cond-4} are in particular satisfied if $u$ is an
$L^2$ normalized eigenfunction of $-\Delta$ for an eigenvalue in the
segment
\[  h^{-2} [1 -C_0h , 1 +C_0h] \,,\]
as applied in Proposition~\ref{theo:sphere}.

\vspace{5mm}

The weaker condition~\eqref{e:original-cond-2} has the advantage that
it is local:
\begin{figure}
\includegraphics{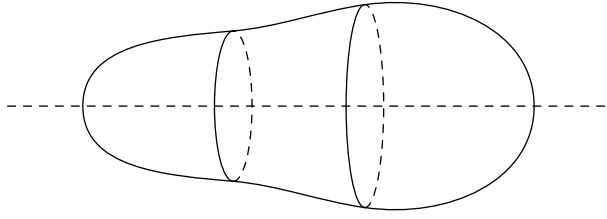}
\caption{A surface of revolution}
\end{figure}
\begin{prop}\label{l:cutoff}
Assume that $\chi\in \CI(M)$ and $D_\varphi\chi=0$.
If $u$ satisfies~\eqref{e:original-cond-1}--\eqref{e:original-cond-3}, then
$\chi u$ satisfies these conditions as well, possibly with larger
value of the constant $C_0 $. 

Similarly, if condition \eqref{e:original-cond-4} holds for $u$, it
holds for $ \chi u $.
\end{prop}
\begin{proof}
Conditions~\eqref{e:original-cond-1} and~\eqref{e:original-cond-3} for
$\chi u$ are trivially satisfied; we now
verify~\eqref{e:original-cond-2}.  Since the commutator
$[-h^2\Delta,\chi]$ is equal to $h$ times a semiclassical differential
operator of order 1, we have
$$
\begin{gathered}
\|(-h^2\Delta-1)\chi u\|_{L^2}\leq
\|\chi(-h^2\Delta-1)u\|_{L^2}+\|[-h^2\Delta,\chi]u\|_{L^2}\\=
\mathcal O(h(1+\|u\|_{H^1_h})).
\end{gathered}
$$
However,
$$
\|u\|_{H^1_h}^2\sim ((-h^2\Delta+1)u,u)_{L^2}={\mathcal O}(1)
$$
by~\eqref{e:original-cond-1} and~\eqref{e:original-cond-2}.
Here $ H_h^s $ denotes the semiclassical Sobolev space
defined using the norm $ \| ( I - h^2 \Delta_M)^{ s /2} u \|_{L^2
} $.

To verify~\eqref{e:original-cond-4}, we use that $\chi u=\chi
X(h)u+{\mathcal O}_{\CI}(h^\infty)$; however, if $Y(h)$ is a compactly
microlocalized pseudodifferential operator equal to the identity
microlocally near the wavefront set of $X(h)$ (and thus of $\chi
X(h)$), then $(1-Y(h))\chi u=(1-Y(h))\chi X(h)u+{\mathcal
O}_{\CI}(h^\infty)={\mathcal O}_{\CI}(h^\infty)$.

\end{proof}
Proposition~\ref{l:cutoff} implies that, if we want to obtain weighted
$L^\infty$ (or any other local) estimates on every function $u$
satisfying~\eqref{e:original-cond-1}--\eqref{e:original-cond-3}, it is
enough to cover $M$ by open sets invariant under rotation (which we
will call bands) and prove the estimates for functions supported in
each of these bands. The next result provides weighted $L^\infty$
estimates for three common types of behavior of the metric in bands:
\begin{theo}
\label{t:1}
Let $U_\varepsilon\subset M$ be a band given by one of the three cases
below; the small parameter $\varepsilon>0$ characterizes the width of
this band. Then for $\varepsilon>0$ and $h>0$ small enough and some
constant $C$, and for each $\alpha\in h \mathbb Z$ varying in a fixed
compact set, each function $u\in \CI(M)$ supported in $U_\varepsilon$
and satisfying~\eqref{e:original-cond-1}--\eqref{e:original-cond-4}
has the following weighted $L^\infty$ estimates:
\begin{figure}
\includegraphics{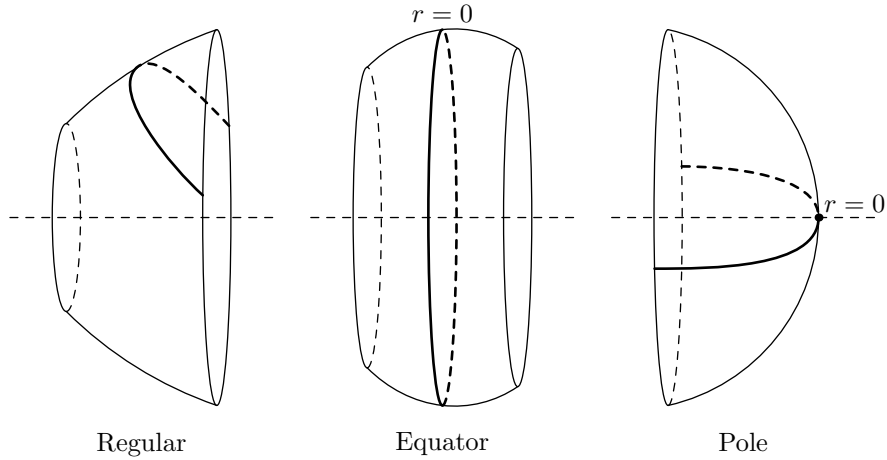}
\caption{Three considered types of bands, with a geodesic shown for each type}
\end{figure}
\begin{enumerate}
\item Regular case: $U_\varepsilon$ has coordinates
$(r,\varphi)\in (-\varepsilon,\varepsilon)\times \mathbb S^1$ and the
metric
\begin{equation}\label{e:regular-band}
g=g(r)[dr^2+d\varphi^2].
\end{equation}
Here $g$ is a smooth function independent of $\varepsilon$, and
$g(0)>0$, $g'(0)\neq 0$. The corresponding estimate is
\begin{equation}\label{e:regular-est}
|u(r,\varphi)|\leq Ch^{-1/6}.
\end{equation}
\item Elliptic equator: $U_\varepsilon$ has the same
coordinates and metric as in the regular case, but $g(0)>0$,
$g'(0)=0$, and $g''(0)<0$. The corresponding estimates are
\begin{gather}
\label{e:equator-est-1}
|u(r,\varphi)|\leq C\min(h^{-1/4}\,, \, h^{-1/6}r^{-1/6}  \,, \,  h^{-1/6}|g(0)-\alpha^2|^{-1/12});\\
\label{e:equator-est-2}
|u(r,\varphi)|\leq Cr^{-1/2} \ \text{ for }|g(0)-\alpha^2|<\varepsilon r^2.
\end{gather} 
\item Pole: $U_\varepsilon$ has coordinates
$(x,y)=(r\cos\varphi,r\sin\varphi)$, with $|x|^2+|y|^2<\varepsilon^2$,
and the metric
\begin{equation}\label{e:pole-band}
g=g(x^2+y^2)[dx^2+dy^2]=g(r^2)[dr^2+r^2\,d\varphi^2],
\end{equation}
where $ g $ is a smooth function and  $g(0)>0$. 
The corresponding estimates are
\begin{gather}
\label{e:pole-est-1}
|u(r,\varphi)|\leq C\min(h^{-1/2}\, , \, h^{-1/6}r^{-1/3} \, , \, h^{-1/6}|\alpha|^{-1/3});\\
\label{e:pole-est-2}
|u(r,\varphi)|\leq Cr^{-1/2} \ \text{ for }|\alpha|<\varepsilon r.
\end{gather}
\end{enumerate}
\end{theo}
\medskip\noindent\textbf{Remark.}
The metric of a surface of revolution can be  brought locally  to
the form~\eqref{e:regular-band} or~\eqref{e:pole-band}, with $g(0)>0$,
with no conditions on the derivatives of $g$. Indeed, away from the
poles (points on the surface lying on the axis of rotation), the
metric has the form
$$
(1+\tilde g'(\tilde r)^2)d\tilde r^2+\tilde g(\tilde r)^2\,d\varphi^2,
$$
where $\tilde r$ is the projection onto the axis of rotation and
$\tilde g$ is a positive function giving the profile of the surface.
Making a change of variables $\tilde r\to r$ with $dr/d\tilde r=\tilde
g(\tilde r)^{-1}\sqrt{1+\tilde g'(\tilde r)^2}$, we bring the metric
to the form~\eqref{e:regular-band}. The case of a pole is handled
similarly. As $(\tilde r,\varphi)$ we can use the geodesic
polar coordinates with respect to the pole and a different change of
variables $\tilde r\to r$.
\medskip

We obtain the following corollary for convex surfaces of revolution. 
The estimate is the analogue of the estimate \eqref{eq:ts1} in the
case of the sphere:

\begin{cor}
\label{cor2}
Suppose that $  M $ is a strictly convex surface of revolution
parametrized by $ ( r, \varphi ) \in [ r_-, r_+ ] \times [0 , 2 \pi )
$, with the metric $dr^2+a(r)\,d\varphi^2$ as in the discussion preceeding
Proposition~\ref{theo:sphere}.
  Suppose that 
\begin{equation}
\label{eq:cond}
 ( \Delta - \lambda  ) u  = 0 , \ \
  ( \textstyle{\frac 1 i } \partial_\varphi - k ) u  = 0 ,\
\|u\|_{L^2(M)}=1. 
\end{equation}
Then, 
\begin{equation}
\label{eq:boundsr}
a ( r )^{1/3} | r - r_0 |^{1/6} |u ( r, \varphi) | \leq C 
\lambda^{1/12} . 
\end{equation}
\end{cor}

The Riemannian volume measure on $ M$ is given by $  a ( r ) dr d \varphi $
which means that the sampling measure based on \eqref{eq:boundsr}
should be given by 
\[         \left( \frac{ a ( r ) } { | r - r_0 | } \right)^{1/3} dr d
\theta ; \]
here $ a ( r ) / ( r-r_0 ) $ is the replacement for $ \tan \theta $. 
The constant $ K $ in \eqref{eq:QkK} is given by 
\[   K \simeq N^{1/12} . \]

Theorem~\ref{t:1} does not cover the case of a band with metric of the
form~\eqref{e:regular-band} and $g'(0)=0$, $g''(0)\leq 0$; in
particular, it does not apply to the case of a hyperbolic equator,
when $g''(0)<0$. This case does not occur for convex surfaces
considered here.

Let us give an informal explanation of the estimates in
Theorem~\ref{t:1}.  By considering an eigenfunction decomposition of
$u$, we can reduce to the case when $u$ is an exact joint
eigenfunction of $-h^2\Delta$ and $hD_\varphi$, rather than a function
satisfying~\eqref{e:original-cond-2}. Using semiclassical analysis
(see~\cite{e-z} for the general theory and the references below for
specific facts we will be using), we can relate the behavior of the
`quantum' object $u$ for small values of the `Planck constant' $h$ to
the corresponding `classical' integrable Hamiltonian system given by
the principal symbols $p$ and $q$ of $-h^2\Delta$ and $hD_\varphi$,
respectively.  The principal symbol of a differential operator is a
polynomial (on each fiber) function on the cotangent bundle $T^*M$,
obtained formally by replacing each instance of $hD_{x_j}$ by the
corresponding momentum $\xi_j$, and then discarding the terms of
higher order in $h$.  In our situation, $p$ is the square of the norm
induced by $g$ on the cotangent bundle and $q$ is the momentum
corresponding to $\varphi$.  The function $u$ will then be
concentrated, or microlocalized, on the set $$
\Lambda=\{p=1,\ q=\alpha\}\subset T^*M.
$$
One can in fact approximate $u$ by certain explicit highly oscillating
integral expressions up to an $\mathcal O(h^\infty)$ error; our
analysis would consist of studying the asymptotic behavior of these
integrals as $h\to 0$.  The set $\Lambda$ consists of unit geodesics
with prescribed angular momentum (that is, of rotations of one such
geodesic); there are two possibilities:
\begin{enumerate}
\item $\Lambda$ is a Lagrangian torus;
\item $\Lambda$ is a circle corresponding to an equator.
\end{enumerate}
In case (1), $u$ is a Lagrangian distribution associated to
$\Lambda$~-- that is, it can be written as a finite sum of expressions
of the WKB form~\eqref{e:oi}, with $\Phi(x,\theta)$ locally
parametrizing the Lagrangian $\Lambda$ and $a$ some smooth symbol~--
see Section~\ref{s:lagr} for details.  The $L^\infty$ norm of $u$
corresponds to how well $\Lambda$ projects onto the base space $M$. At
a point where the tangent space of $\Lambda$ projects surjectively
onto $M$, the $L^\infty$ norm of $u$ is $\mathcal O(1)$. The only
other possibility that could arise in the regular case is a turning
point; that is, a point where the function $r$ has a nondegenerate
critical point when restricted to a geodesic (for the sphere, these
are the points of maximal and minimal latitude on a given great
circle). The behavior of $u$ near the turning points is similar to
that of the Airy function, and its $L^\infty$ norm is of order
$h^{-1/6}$ by a variant of Van der Corput's lemma.

If one is unfamiliar with Lagrangian distributions, the simple model
case to consider would be the eigenfunctions of the Laplacian
$h^2D_x^2$ on the circle $\mathbb S^1=\mathbb R/(2\pi \mathbb
Z)$. Those are given by $e^{i\lambda x/h}$, with the eigenvalue
$\lambda^2$.  (Note that only a discrete set of $\lambda$ is possible
here~-- this is a baby version of the quantization condition mentioned
in the next section.)  The corresponding symbol is $p(x,\xi)=\xi^2$
and the corresponding Lagrangian would be $\{x\in \mathbb S^1,\
\xi=\lambda\}$; it projects surjectively onto the $x$ variables, which
corresponds to the fact that we do not need any integration variables
$\theta$ in the formula~\eqref{e:oi} to define eigenfunctions, and to
the fact that the $L^\infty$ norm of eigenfunctions is bounded by a
constant.

For $\alpha=0$, a new kind of problem arises~-- the intersection of
$\Lambda$ with the fiber of $T^*M$ at a pole is a not a point, but a
circle consisting of all unit covectors at the pole, leading to a loss
of $h^{-1/2}$ in the $L^\infty$ norm. This problem disappears and we
get back the $h^{1/6}$ estimate if either $\alpha$ is away from zero
or we are away from the pole, which is reflected
in~\eqref{e:pole-est-1}. If $\alpha=0$, the turning points are located
at the poles; therefore, away from the poles we get an $\mathcal O(1)$
estimate. The blow-up rate of this estimate as we approach a pole is
quantified by~\eqref{e:pole-est-2}.

In case (2) we can separate out the $\varphi$ variable (as $\Lambda$
does not pass through any poles) and obtain a one-dimensional problem;
then $u$ is a low-lying eigenfunction of a Schr\"odinger operator with
a potential well. The bottom of the well eigenfunctions (those with
$g(0)-\alpha^2=\mathcal O(h)$) are approximated by the Gaussian
$h^{-1/4}e^{-r^2/(2h)}$; we see that they are
$$\mathcal O(\min(h^{-1/4},r^{-1/2})) \,. $$
This explains the first part of~\eqref{e:equator-est-1}
and~\eqref{e:equator-est-2}. However, if $g(0)-\alpha^2$ is bounded
away from zero, we are away from the equator and thus back to case
(1), which explains the $h^{-1/6}$ term in~\eqref{e:equator-est-1}.

\section{Preliminaries}
\label{prel}

\subsection{Semiclassical Lagrangian distributions}\label{s:lagr}

In this subsection, we briefly review the local theory of
semiclassical Lagrangian distributions; see for example \cite{A},\cite[Chapter~6]{g-s},
\cite[Chapter~8]{g-s2} or~\cite[Section~2.3]{svn} for a detailed account,
and~\cite[Section~25.1]{h4} or~\cite[Chapter~11]{gr-s}
for the presentation in the closely related microlocal case.

Assume that $X$ is a manifold and $\Phi(x,\theta)$ is a smooth real-valued function
defined on an open set $U_\Phi\subset X\times \mathbb R^m$. Define
the critical set $C_\Phi\subset U_\Phi$ by
$$
(x,\theta)\in C_\Phi\iff \partial_\theta\Phi(x,\theta)=0.
$$
The function $\Phi$ is called
a (nondegenerate) phase function, if for each $(x,\theta)\in C_\Phi$,
the differentials $d(\partial_{\theta_1}\Phi),\dots,d(\partial_{\theta_m}\Phi)$
are linearly independent at $(x,\theta)$. If this is the case, the set
$$
\Lambda_\Phi=\{(x,\partial_x\Phi(x,\theta))\mid (x,\theta)\in C_\Phi\}\subset T^*X
$$
is an (immersed) Lagrangian submanifold. We say that $\Phi$ generates
$\Lambda_\Phi$; in general, if $\Lambda\subset T^*X$ is a Lagrangian
submanifold and $(x,\xi)\in\Lambda_\Phi\subset\Lambda$, then we say
that $\Phi$ generates $\Lambda$ near $(x,\xi)$. For each Lagrangian
submanifold $\Lambda$ and each $(x,\xi)\in\Lambda$, there exists a
phase function generating $\Lambda$ near $(x,\xi)$; however, such
phase function is not unique.

If $\Lambda\subset T^*X$ is an (embedded) Lagrangian submanifold,
$\Phi$ is a phase function generating $\Lambda$ near some point, and
$a(x,\theta;h)\in \CIc(U_\Phi)$ is bounded in $\CI_{x,\theta}$
uniformly in $h$, we can define the $h$-dependent family of smooth
functions
\begin{equation}\label{e:oi}
u(x;h)=h^{-m/2}\int e^{i\Phi(x,\theta)/h}a(x,\theta)\,d\theta.
\end{equation}
Here the factor $h^{-m/2}$ is chosen so that $\|u(x;h)\|_{L^2}$ is
bounded by a certain $\CI$ seminorm of $a$ and equivalent to the
$L^2$ norm of $a|_{C_\Phi}$, modulo ${\mathcal O}(h)$ terms.  We
call $u(x;h)$ a (semiclassical compactly microlocalized) Lagrangian
distribution associated to $\Lambda$. This family is microlocalized on
$\Lambda_\Phi$ in the following sense:
\[  b \in \CIc  (T^*  X) \,, \
b |_{\neigh (\Lambda_\Phi) } = 0
\ \Longrightarrow \ 
\|b^w(x,hD_x)u(x;h)\|_{L^2}={\mathcal O}(h^\infty)\,. \]
More generally, we call $u(x;h)$ a Lagrangian distribution associated
to $\Lambda$, if it is the sum of finitely many expressions of the
form~\eqref{e:oi}, for different phase functions parametrizing
$\Lambda$, and an ${\mathcal O}_{\CIc}(h^\infty)$ remainder.

If $u(x;h)$ is a Lagrangian distribution associated to $\Lambda$,
$\Phi$ is a phase function generating $\Lambda$, and $u$ is
microlocalized in a compact subset of $\Lambda_\Phi$, then we can
always write $u$ in the form~\eqref{e:oi} for some amplitude $a$
modulo an ${\mathcal O}(h^\infty)$ remainder. (The general case can
always be reduced to this one by a microlocal partition of unity, if
we know enough phase functions to cover the whole $\Lambda$.) In other
words, if two phase functions locally generate the same Lagrangian,
then oscillatory integrals~\eqref{e:oi} associated to one phase
function can be written in the form~\eqref{e:oi} using the other phase
function as well. However, the formulas relating even the principal
parts of the amplitudes corresponding to different phase functions are
quite complicated; obtaining a geometric interpretation for these
formulas is the subject of global theory of Lagrangian distributions.
This global theory is needed to produce quantization conditions that
we use below; however, as we are only interested in estimating the
resulting eigenfunctions and in some rough properties of the spectrum
given by quantization conditions, we do not use the global theory
directly.

\subsection{Specific generating functions and $L^\infty$ estimates}

In this subsection, we assume that $\Lambda\subset T^*\mathbb R^n$
is a Lagrangian submanifold and $(x^0,\xi^0)\in\Lambda$.
\begin{prop}\label{l:generating}
Take $0\leq m\leq n$ and denote $x=(x',x''),\xi=(\xi',\xi'')$,
where $x',\xi'\in \mathbb R^m$ and $x'',\xi''\in \mathbb R^{n-m}$.
Assume that $T_{(x^0,\xi^0)}\Lambda$ projects surjectively
onto the $(x',\xi'')$ variables. Then there exists a function
$S(x',\xi'')$ such that near $(x^0,\xi^0)$, $\Lambda$ is given by
$$
\{\xi'=-\partial_{x'}S(x',\xi''),\
x''=\partial_{\xi''}S(x',\xi'')\}.
$$
Consequently, $\Lambda$ is generated near $(x^0,\xi^0)$
by the phase function
$$
\Phi(x,\theta)=x''\cdot\theta-S(x',\theta),\
\theta\in \mathbb R^{n-m}.
$$
\end{prop}
\begin{proof}
For the reader's convenience we recall the well known argument.
We can write $\Lambda$ locally as a graph $\{\xi'=-F(x',\xi''),\
x''=G(x',\xi'')\}$.  The restriction of the symplectic form
$d\xi\wedge dx$ to $\Lambda$ is zero; therefore, the restriction of
the 1-form $\alpha=x''\,d\xi''-\xi'\,dx'$ to $\Lambda$ is closed.
Therefore, there exists a function $S(x',\xi'')$ such that $\alpha=dS$
when restricted to $\Lambda$. However, if we use $(x',\xi'')$ as a
coordinate system on $\Lambda$, then $\alpha=G\,d\xi''+F\,dx'$ and
$dS=\partial_{\xi''}S\,d\xi''+\partial_{x'}S\,dx'$; it follows that
$F=\partial_{x'}S$ and $G=\partial_{\xi''}S$.
\end{proof}
In one special case $\Lambda$ can be locally parametrized by $x$; the
corresponding Lagrangian distributions satisfy the best $L^\infty$
estimate possible:
\begin{prop}\label{l:horizontal}
Assume that the tangent space to $\Lambda$ at each point projects
surjectively onto the $x$ variables. Then, each Lagrangian
distribution $u(x;h)$ associated to $\Lambda$ satisfies
$\|u(x;h)\|_{L^\infty}={\mathcal O}(1)$.
\end{prop}
\begin{proof}
Putting $x'=x$ and $x''=\emptyset$ in Proposition~\ref{l:generating},
we get a phase function $\Phi(x)=-S(x)$.  The expression~\eqref{e:oi}
for this phase function has the form $e^{-iS(x)/h}a(x;h)$ and is
trivially ${\mathcal O}(1)$ in $L^\infty$.
\end{proof}
If $\Lambda$ does not satisfy the condition of
Proposition~\ref{l:horizontal}, then the associated Lagrangian
distributions can have $L^\infty$ norm as large as
$h^{-n/2}$. However, in some cases we are still able to find a phase
function satisfying some additional conditions that will ensure a
better bound.  We will in particular use the following estimate for
the case when $\theta$ is one-dimensional:
\begin{prop}\label{l:3rd-derivative-estimate}
Assume that $\Phi(x,\theta)$, $(x,\theta)\in U\subset \mathbb R^n\times \mathbb R$,
is a phase function and
$$
|\partial_\theta\Phi(x,\theta)|+|\partial_\theta^2\Phi(x,\theta)|
+|\partial_\theta^3\Phi(x,\theta)|>0\text{ for each }
(x,\theta)\in U.
$$
Then each oscillatory integral $u(x;h)$ of the form~\eqref{e:oi} satisfies
$$
\|u\|_{L^\infty}={\mathcal O}(h^{-1/6}).
$$
\end{prop}
\begin{proof}
This follows from the van der Corput's Lemma as presented
in \cite[Proposition VIII.2]{Stein}. We can also argue
using more precise asymptotics which are relevant when 
precise information at the caustic is needed:
If the amplitude $a(x,\theta;h)$ is supported in a region where
$|\partial_\theta\Phi|+|\partial_\theta^2\Phi|>0$, then we can apply
the stationary phase method 
to~\eqref{e:oi} to get $\|u\|_{L^\infty}={\mathcal
O}(1)$.  (Alternatively, the corresponding piece of the Lagrangian
$\Lambda_\Phi$ will satisfy the condition of
Proposition~\ref{l:horizontal}.) If at  $ (x_0, \theta_0
) $ we have  $ \partial_\theta \Phi ( x_0, \theta_0 ) =
\partial^2_\theta \Phi ( x_0, \theta_0 ) = 0 $ and $ \partial^3_\theta
\Phi ( x_0 , \theta_0 ) \neq 0 $, then we can apply \cite[Theorem
7.7.18]{h1} which shows that for $ a $ supported near $ (x_0 ,
\theta_0 ) $, there exist smooth functions $a_1,a_2,b$ such that
$a_1,b$ are real-valued and
\[  h^{-\frac 12} \int  e^{ i \Phi  ( x , \theta ) /
  h } a ( x, \theta ) \, d \theta = h^{-1/6} e^{ i b ( x ) / h } Ai ( h^{-\frac23} a_1 ( x
) ) a_2 ( x ) + {\mathcal O} ( h^{\frac16} ) \,. \]
Since the Airy function $Ai$ is bounded, this expression is $\mathcal
O(h^{-1/6})$.  A partition of unity argument finishes the proof of the
proposition.
\end{proof}

\subsection{$L^\infty$ estimates for potential wells}\label{s:well}

Assume that $V(x)$, $x\in \mathbb R$, is a smooth real-valued potential such that:
\begin{itemize}
\item for each $N$, there exists a constant $C_N$ such that
$|\partial_x^N V(x)|\leq C_N(1+|x|^2)$ for all $x$;
\item there exists a constant $\varepsilon_V>0$ such that
$V(x)\geq\varepsilon_V |x|^2$ for all $x$;
\item $V(0)=V'(0)=0$, $V''(0)=2c^2$ for some $c>0$, and
$\pm V'(x)>0$ for $\pm x>0$.
\end{itemize}
Consider the semiclassical Schr\"odinger operator
$$
P(h)=h^2D_x^2+V(x);
$$
let $p(x,\xi)=\xi^2+V(x)$ be the corresponding principal symbol.
Since $p(x,\xi)\to\infty$ as $|x|+|\xi|\to\infty$, it
follows~\cite[Section~6.3]{e-z} that $P(h)$ is a self-adjoint operator
on $L^2(\mathbb R)$ with compact resolvent and therefore posesses a
complete orthonormal system of eigenfunctions.

For $\lambda>0$, $p$ does not have critical points on
$p^{-1}(\lambda)$ and thus produces a completely integrable
one-dimensional Hamiltonian system.  The following is a corollary of
the spectral theory of self-adjoint operators corresponding to quantum
completely integrable systems (see~\cite[Section~2.7]{g-s}
or~\cite[Theor\`eme~5.1.11]{svn}):
\begin{prop}\label{l:well-qc}
Let $K$ be a compact subset of $(0, \infty ) $. Then for $h$ small enough,
the eigenvalues
$\lambda$ of $P(h)$ in $K$ are simple and
given by the quantization condition
\begin{equation}\label{e:well-qc}
S(\lambda)=(2j+1)\pi h+{\mathcal O}(h^2),\
j\in \mathbb Z,\
\end{equation}
where $S(\lambda)$ is the action functional:
\[  S ( \lambda ) =  \oint_{ p ( x , \xi ) = \lambda } \xi\, dx \]
and the curve $ p^{-1} ( \lambda) $ is oriented clockwise.  In
particular, $S$ is smooth and $S'(\lambda)>0$ everywhere.  Moreover,
each $L^2$ normalized joint eigenfunction corresponding to an
eigenvalue $\lambda\in K$ is a Lagrangian distribution associated to
$p^{-1}(\lambda)$ (that is, it admits a parametrization by integrals
of the form~\eqref{e:oi} with estimates on the symbol and the
remainder uniform in $\lambda$).
\end{prop}

\medskip
\noindent
\textbf{Remark.} The quantization condition~\eqref{e:well-qc}
is valid for the eigenvalues close to zero as well, if we add the
requirement that $j\geq 0$. This can be proved by conjugating $P(h)$
by a semiclassical Fourier integral operator to a function of the
harmonic oscillator; it follows from the analysis in~\cite{cdv} and
can be found for example in~\cite[Theor\`eme~5.2.4]{svn}.
\medskip

For $\lambda$ away from zero, we get the following
estimate on eigenfunctions of $P(h)$:
\begin{prop}\label{l:well-eig}
Let $K$ be a compact subset of $(0,\infty)$.
Then each
$L^2$-normalized eigenfunction $v_\lambda$ with eigenvalue $\lambda\in K$
satisfies
$$
\|v_\lambda\|_{L^\infty}={\mathcal O}(h^{-1/6}).
$$
\end{prop}
\begin{proof}
\begin{figure}
\includegraphics{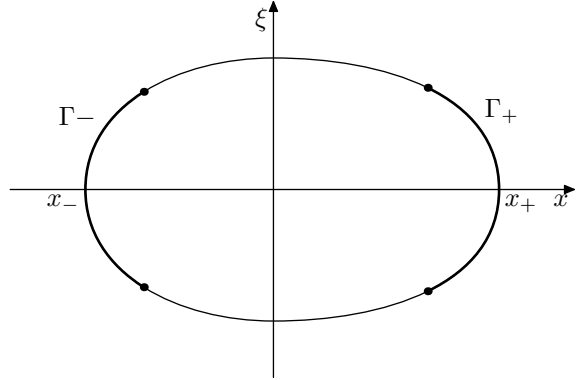}
\caption{The curve $p^{-1}(\lambda)$ and the regions $\Gamma_\pm$}
\end{figure}
By Proposition~\ref{l:well-qc}, $v_\lambda$ is a Lagrangian
distribution associated to the curve $p^{-1}(\lambda)$.  Let
$\Gamma_\pm$ be the intersections of $p^{-1}(\lambda)$ with small
neighborhoods of the turning points $(x_\pm(\lambda),0)$; here
$x_-(\lambda)<x_+(\lambda)$ are the two roots of the equation
$V(x)=\lambda$. Away from $\Gamma_\pm$, the curve $p^{-1}(\lambda)$
projects diffeomorphically onto the $x$ axis; therefore, by
Proposition~\ref{l:horizontal}, Lagrangian distributions associated to
$p^{-1}(\lambda)\setminus (\Gamma_+\cup\Gamma_-)$ have $L^\infty$ norm
${\mathcal O}(1)$.  It remains to consider the case of a Lagrangian
distribution associated to, say, $\Gamma_+$.  However, $\Gamma_+$ can
be parametrized by $\xi$ and thus by Proposition~\ref{l:generating} it
has a generating function of the form $\Phi(x,\xi)=x\xi-T(\xi)$. We
can calculate that $\partial_\xi^3 T(\xi)$ is nonvanishing; it remains
to apply Proposition~\ref{l:3rd-derivative-estimate}.
\end{proof}
Using Proposition~\ref{l:well-eig} together with elliptic estimates,
we get $L^\infty$ bounds on functions satisfying the analogue
of~\eqref{e:original-cond-2}:
\begin{prop}\label{l:well-est} 
Assume $\lambda\in \mathbb R$ lies in a fixed compact set $ K $ and
$u\in \mathcal \CIc(\mathbb R)$ satisfies%
\begin{equation}\label{e:well-ae}
\|u\|_{L^2}\leq C_0,\
\|(P(h)-\lambda)u\|_{L^2}\leq C_0 h,
\end{equation}
for some constant $C_0$. Then:
\begin{enumerate}
\item If $\chi\in \CIc(T^*\mathbb R)$ is equal to 1 near
$p^{-1}(\lambda)$, then
$$
\|(1-\chi^w(x,hD_x))u\|_{L^\infty}=\mathcal O_{\chi,\lambda}(1).
$$
In particular, if $\pi:T^* \mathbb R\to \mathbb R$ is the projection
map, then
\begin{equation}
\label{eq:eps}   x\not\in \pi ( p^{-1} ( \lambda))
\Longrightarrow \ 
| u ( x ) | = \mathcal O_{x,\lambda} ( 1 ) \,. \end{equation}
 The constants in $\mathcal O(\cdot)$ are locally uniform in
$x,\lambda$.

\item If $|\lambda|\geq \delta$ for some constant $\delta>0$, then
$$
\|u\|_{L^\infty}={\mathcal O}_\delta(h^{-1/6}).
$$
\item  If $|x|\geq \delta$ for some constant $\delta>0$, then
$$ | u ( x ) | ={\mathcal O}_\delta(h^{-1/6}). $$
\end{enumerate}
\end{prop}

\medskip\noindent
{\bf Remark.} We do not assume the localization condition
\eqref{e:original-cond-4} and work only under the hypothesis
\eqref{e:well-ae}. That is because we will apply the proposition to a
rescaled version of the original $ u $ which no longer satisfies
\eqref{e:original-cond-4}.
\medskip

\begin{proof}
1. The operator $P(h)-\lambda$ is elliptic with respect to the order
function $m(x,\xi)=1+x^2+\xi^2$ on $\supp(1-\chi)$; therefore, the
elliptic parametrix construction (as in~\cite[Section~4.5]{e-z}
or~\cite[Proposition~8.6]{d-s}, with minor adjustments found for
example in~\cite[Proposition~5.1]{d}) gives a symbol $q\in S(m^{-1})$
such that
$$
q^w(x,hD_x)(P(h)-\lambda)=1-\chi^w(x,hD_x)+{\mathcal O}_{L^2 \to L^2} (h^\infty).
$$
Therefore, by Sobolev embedding
\[ 
\begin{split}
\|(1-\chi^w(x,hD_x))u\|_{L^\infty}
& \lesssim \|(1-\chi^w(x,hD_x))u\|_{H^1} 
\lesssim h^{-1}\|(1-\chi^w(x,hD_x))u\|_{H^1_h}\\
& \lesssim h^{-1}\|q^w(x,hD_x)(P(h)-\lambda)u\|_{H^1_h} + {\mathcal O}
( h^\infty ) \| u \|_{L^2} \\
& \lesssim h^{-1}\|(P(h)-\lambda)u\|_{L^2} + {\mathcal O}
( h^\infty ) \| u \|_{L^2}  ={\mathcal O}(1).
\end{split}
\]
To show~\eqref{eq:eps}, it suffices to apply the previous inequality
with $\chi$ vanishing near $\pi^{-1}(x)$, but equal to 1 near
$p^{-1}(\lambda)$.

2. If $\lambda\leq-\delta$, then $p^{-1}(\lambda)=\emptyset$ and thus
part~1 of this proposition applies with $\chi \equiv 0$. We now assume that
$\lambda\geq\delta$.  Take $\tilde\chi\in \CIc(\lambda/2,2\lambda)$
equal to 1 near $\lambda$; then the operator $\tilde\chi(P(h))$,
defined by means of functional calculus, is pseudodifferential,
compactly microlocalized, and equal to the identity microlocally near
$p^{-1}(\lambda)$~\cite[Chapter~8]{d-s}.  By part~1 of this
proposition, we have
$$
u=\tilde\chi(P(h))u+{\mathcal O}_{L^\infty}(1);
$$
it remains to estimate $\|\tilde\chi(P(h))u\|_{L^\infty}$. Let
$v_1,v_2,\ldots\in \CI$ be the orthonormal basis of $L^2(\mathbb R)$
consisting of eigenfunctions of $P(h)$ with eigenvalues
$\lambda_1,\lambda_2,\dots$. If
$$
u=\sum_j c_j v_j,
$$
then by~\eqref{e:well-ae},
$$
\begin{gathered}
\sum_j (|\lambda_j-\lambda|+h)^2|c_j|^2\leq 4C_0^2h^2;\\
\|\tilde\chi(P(h))u\|_{L^\infty}\leq \sum_{\lambda/2\leq\lambda_j\leq 2\lambda}|c_j|\cdot\|v_j\|_{L^\infty}\\
\leq 2C_0 h \Big(\sup_{\lambda/2\leq\lambda_j\leq 2\lambda}\|v_j\|_{L^\infty}\Big)\cdot
\bigg(\sum_{\lambda/2\leq\lambda_j\leq 2\lambda}(|\lambda_j-\lambda|+h)^{-2}\bigg)^{1/2}.
\end{gathered}
$$
The sum in the second factor is ${\mathcal O}(h^{-1})$ by the quantization
condition~\eqref{e:well-qc}.  The supremum in the first factor is 
${\mathcal O}(h^{-1/6})$ by Proposition~\ref{l:well-eig}.

3. There exists $\delta_1>0$, depending on $\delta$, such that if
$|x|\geq\delta$ and $|\lambda|\leq\delta_1$, then
$x\not\in\pi(p^{-1}(\lambda))$. The case $|\lambda|\leq\delta_1$ is
then handled by~\eqref{eq:eps}, while the case $|\lambda|\geq\delta_1$
is handled by part~2 of this proposition.
\end{proof}

\section{Proof of Theorem \ref{t:1}}
\label{proo}

\subsection{Regular case}

We have
$$
g(r)[-h^2\Delta-1]=h^2[D_r^2+D_\varphi^2]-g(r).
$$
We then separate out the $\varphi$ variable:
if $u(r,\varphi)=u(r)e^{i\alpha\varphi/h}$, then
\begin{equation}\label{e:sep-cond}
u(r)\in \CIc(-\varepsilon,\varepsilon)\, ,\ \ 
\|u\|_{L^2}={\mathcal O}(1) \,,\ \ 
\|(h^2D_r^2+\alpha^2-g(r))u\|_{L^2}={\mathcal O}(h) \,. 
\end{equation}
Now, for $\varepsilon$ small enough,
\[ g(r)=V_0 - V(r+x_0) \,,  \ \ |r|<\varepsilon\, \]
where a potential $V(x)$ satisfies conditions of Section~\ref{s:well}.
Moreover, $|x_0|>2\varepsilon$. If $\tilde u(x)=u(x-x_0)$, then
$\|(P(h) - V_0+\alpha^2)\tilde u\|_{L^2}={\mathcal O}(h)$, where
$$
P(h)=h^2D_x^2+V(x).
$$
However, $\tilde u$ is supported away from zero; therefore, by part~3 of Proposition~\ref{l:well-est},
we find
$$
\|u\|_{L^\infty}={\mathcal O}(h^{-1/6})
$$
as required.

\subsection{Case of an equator}

As in the regular case, we can separate out the $\varphi$ variable and
obtain~\eqref{e:sep-cond}.  Now, take a potential $V(r)$ satisfying
conditions of Section~\ref{s:well} such that $V(r)=g(0)-g(r)$ for
$|r|<\varepsilon$. Then, $\|(P(h)-\lambda)u\|_{L^2}={\mathcal O}(h)$,
with
$$
P(h)=h^2D_r^2+V(r) \,, \ \ \
\lambda=g(0)-\alpha^2.
$$
For $\lambda$ or $r$ bounded away from zero, we can argue similarly to
the regular case. Indeed, \eqref{e:equator-est-1} follows from parts~2
and~3 of Proposition~\ref{l:well-est}.  To obtain the estimate
\eqref{e:equator-est-2} for $r$ bounded away from zero, we
use~\eqref{eq:eps}: the condition in
\eqref{e:equator-est-2} means that $ |\lambda| \leq \varepsilon r^2 $ so
that $r\not\in\pi(p^{-1}(\lambda))$.

If both $\lambda$ and $r$ are close to zero, we will use the natural
rescaling of the quantum harmonic oscillator: for some constant $C_1$
and $C_1^{-1}h\leq\beta\leq C_1$, define
$$
\begin{gathered}
(T_\beta f)(\tilde r)=\beta^{1/4}f(\beta^{1/2}\tilde r),\ f\in L^2(\mathbb R);\\
h_\beta=\beta^{-1}h,\
\lambda_\beta=\beta^{-1}\lambda \,,  \ \  \\
V_\beta(\tilde r)=\beta^{-1}V(\beta^{1/2}\tilde r), \ \ 
P_\beta(h_\beta)=h_\beta^2D_{\tilde r}^2+V_\beta(\tilde r).
\end{gathered}
$$
We note that if $ V ( r ) = r^2 $, then 
\[ P_\beta ( h_\beta ) =  h_\beta^2 D_{\tilde r }^2 + V ( \tilde r )
\,; \]
that is, the operator does not change.  For a general potential we get
closer to the harmonic oscillator as $ \beta \rightarrow 0
$. Moreover, the potential $V_\beta$ satisfies conditions of
Section~\ref{s:well} uniformly in $\beta$. Indeed, the only nontrivial
part is verifying the first of these conditions for $N=0,1$, and this
follows from the fact that $V(r)={\mathcal O}(r^2)$ and
$V'(r)={\mathcal O}(|r|+r^2)$. Therefore, the constants in the
estimates of Proposition~\ref{l:well-est} do not depend on $\beta$.

The operator  $T_\beta$ is unitary  on $L^2(\mathbb R)$ and
$$
P_\beta( h_\beta)-\lambda_\beta=\beta^{-1}T_\beta
(P(h)-\lambda)T_\beta^* \,, 
$$
Therefore, if $u_\beta=T_\beta u$, then
$$
\|u_\beta\|_{L^2}={\mathcal O}(1),\
\|(P_\beta(h_\beta)-\lambda_\beta)u_\beta\|_{L^2}={\mathcal O}(h_\beta).
$$
For $\varepsilon>0$ small enough, we have $P(h)\geq
\varepsilon(h^2D_r^2+ r^2)$; the ground state of the quantum harmonic
oscillator~\cite[Section~6.1]{e-z} then gives $P(h)\geq\varepsilon h$.
Therefore, we can assume that $|\lambda|\geq\varepsilon h$~-- indeed,
if $|\lambda|\leq\varepsilon h$, then we can replace $\lambda$ by
$\varepsilon h$ and obtain same conditions on $u$ and stronger
conclusions (except for~\eqref{e:equator-est-2} in case when
$r^2<h$~-- however, in this case $r^{-1/2}\geq h^{-1/4}$
and~\eqref{e:equator-est-2} follows from~\eqref{e:equator-est-1}).
Now, consider the following two subcases:
\begin{enumerate}
\item $|\lambda|<\varepsilon r^2$. Since $|\lambda|\geq\varepsilon h$,
we have $| r| >h^{1/2}$. Put $\beta=r^2$; then $u(r)=r^{-1/2}u_\beta(\pm
1)$ (depending on the sign or $r$) and $|\lambda_\beta|<\varepsilon$.
For $\varepsilon$ small enough, $\tilde r=\pm 1$ does not lie in the
projection of $\{\xi^2+V_\beta(\tilde r)=\lambda_\beta\}$; therefore,
\eqref{eq:eps} applies, giving $|u_\beta(\pm 1)|={\mathcal O}(1)$ and
thus $|u(r)|={\mathcal O}(r^{-1/2})$; this
proves~\eqref{e:equator-est-2}, which implies~\eqref{e:equator-est-1}
in the present case.
\item $|\lambda|\geq\varepsilon r^2$, where $\varepsilon$ is chosen as
in case (1).  Put $\beta=|\lambda|$; then $\lambda_\beta=\pm 1$
depending on the sign of $\lambda$ and
$u(r)=|\lambda|^{-1/4}u_\beta(\tilde r)$, where $\tilde
r=|\lambda|^{-1/2}r$ is bounded. Now, part~2 of
Proposition~\ref{l:well-est} applies for $\lambda_\beta=1$ and part~1
of the same proposition applies for $\lambda_\beta=-1$ (with
$\chi=0$), giving $|u_\beta(\tilde r)|={\mathcal O}(h_\beta^{-1/6})$
and thus $|u(r)|={\mathcal O}(h^{-1/6}|\lambda|^{-1/12})$.  This
proves~\eqref{e:equator-est-1} as the third term in the minimum
expression is controlled by each of the other two in the present case.
\end{enumerate}

\subsection{Case of a pole}

Define the operators $P(h)=-h^2\Delta$ and $Q(h)=hD_\varphi$; let
$p,q$ be the corresponding semiclassical principal symbols; for
$\lambda\geq 0$, consider the compact set
$$
\Lambda(\lambda,\alpha)=\{p=\lambda,\ q=\alpha\}\subset T^*M.
$$
We start by an analogue of part~1 of Proposition~\ref{s:well}:
\begin{prop}\label{l:elliptic-pole}
Assume that $X(h)$ is a compactly microlocalized pseudodifferential
operator on $\mathbb S^2$ microlocally equal to the identity near
$\Lambda(1,\alpha)$.  Then
$$
\|(1-X(h))u\|_{L^\infty}={\mathcal O}(1).
$$
\end{prop}
\begin{proof}
Using the localization assumption~\eqref{e:original-cond-4}, a
microlocal partition of unity, and the elliptic parametrix
construction (see the proof of Proposition~\ref{l:well-est}) we can
write
\begin{equation}
\label{eq:Xh}
(1-X(h))u=A(h)(P(h)-1)u+B(h)(Q(h)-\alpha)u+{\mathcal O}_{\CI}(h^\infty),
\end{equation}
where both $A(h)$ and $B(h)$ are pseudodifferential and compactly microlocalized,
and thus act $L^2\to L^\infty$ with norm ${\mathcal O}(h^{-1})$~\cite[Theorem~7.10]{e-z}.
We now  use  ~\eqref{e:original-cond-2} and~\eqref{e:original-cond-3}
in \eqref{eq:Xh} and that proves the proposition.
\end{proof}
If $(\xi,\eta)$ are the momenta corresponding to the coordinates $(x,y)=(r\cos\varphi,r\sin\varphi)$,
then
$$
p=g(x^2+y^2)^{-1}(\xi^2+\eta^2),\
q=x\eta-y\xi.
$$
Therefore, if $\alpha$ is bounded away from zero and $\varepsilon$ is
small, then the projection of $\Lambda(1,\alpha)$ onto $M$ does not
intersect $U_\varepsilon=\{r<\varepsilon\}$ and
$\|u\|_{L^\infty}={\mathcal O}(1)$ by
Proposition~\ref{l:elliptic-pole}. Therefore, we may assume that
$|\alpha|<\varepsilon$. The symbols $p$ and $q$ have linearly
independent differentials on $\Lambda(1,0)$, and they Poisson commute;
therefore, $\Lambda(\lambda,\alpha)$ is a Lagrangian torus for small
$\alpha$ and $\lambda$ close to 1. The joint spectrum of $P(h),Q(h)$
near $(1,0)$ obeys a quantization condition; in particular, the
spectrum of $P(h)$ restricted to the eigenspace $\{Q(h)=\alpha\}$ is
approximated by a formula similar to~\eqref{e:well-qc}. (See~\cite{ch}
or~\cite[Theor\`eme~5.1.11]{svn}.) Similarly to the proof of part~2 of
Proposition~\ref{l:well-est}, we reduce the problem to the following
\begin{prop}
If $\varepsilon>0$ is small enough, then
each $L^2$ normalized joint eigenfunction $v$ for $P(h),Q(h)$ with eigenvalue
$(\lambda,\alpha)$, where $|\lambda-1|,|\alpha|<\varepsilon$,
satisfies the following estimates for $|r|<\varepsilon$:
\begin{gather}
\label{e:pole-est2-1}
|v(r,\varphi)|={\mathcal O}(\min(h^{-1/2},h^{-1/6}r^{-1/3},h^{-1/6}|\alpha|^{-1/3}));\\
\label{e:pole-est2-2}
|v(r,\varphi)|={\mathcal O}(r^{-1/2})\text{ for }|\alpha|<\varepsilon r.
\end{gather}
\end{prop}
\begin{proof}
As follows from the proof of the quantization condition,
$v(r,\varphi)$ is a Lagrangian distribution associated to
$\Lambda(\lambda,\alpha)$.  Using a microlocal partition of unity and
the rotational symmetry of the problem, we can reduce to the case when
$v$ is microlocalized in $\{\xi>\eta/2,\ |r|<\varepsilon\}$.  However,
in this region $\Lambda(\lambda,\alpha)$ can be parametrized by
$x,\eta$, since the matrix of derivatives of $p,q$ in $y,\xi$ is
nondegenerate for $x=y=0$; by Proposition~\ref{l:generating},
$\Lambda(\lambda,\alpha)$ is generated by a phase function of the form
$\Phi=y\eta-S(x,\eta;\lambda,\alpha)$, where  $S$ is a smooth
function. We recall that this means that on $ p = \lambda $ and
$ q = \alpha $, $ \xi = -\partial_x S $ and $ y = \partial_\eta S $.

From this we calculate
\[ 
\begin{split}
& \partial_x S(0,\eta,\lambda,\alpha)=  - \sqrt {g (y^2) \lambda  -
  \eta^2} = -\sqrt{\lambda g(0)-\eta^2}+{\mathcal O}(\alpha^2)\,,\\
& \partial_\eta S(0,\eta,\lambda,\alpha)= - \alpha / \xi = 
- \frac{\alpha }{\sqrt{\lambda g(0)-\eta^2}}+{\mathcal O}(\alpha^2);
\end{split}
\]
if we normalize $S$ by the condition $S(0,0,\lambda,\alpha)=0$, then
$
S(0,\eta , \lambda,0)=0 $, and 
\[ \begin{split}  \partial_\alpha S(0,\eta,\lambda,0) & =  \partial_\alpha S ( 0 , 0, 
\lambda , 0 ) + \int_0^\eta \partial_\alpha \partial_\eta S ( 0 ,
t, \lambda, 0 ) dt = - \int_0^\eta \frac {dt}  {\sqrt {\lambda g ( 0 ) -
  t^2 } } \\
& = - \arcsin ( \eta / { \sqrt { \lambda g ( 0)} } ) \,. 
\end{split} \]
Now, put $(r,\alpha)=(s\tilde r,s\tilde\alpha)$ with $\tilde
r^2+\tilde\alpha^2=1$ and define $\tilde h=h/s$; we get
by~\eqref{e:oi},
\[
\begin{split}
v(r,\varphi;\lambda,\alpha,h) & =h^{-1/2}\int \exp 
\left( {i(s\tilde r\eta\sin\varphi-S(s\tilde
    r\cos\varphi,\eta;\lambda,s\tilde\alpha))/h}  \right) a\,d\eta\\
& =s^{-1/2}\tilde h^{-1/2}\int \exp \left( {i\tilde\Phi(\eta,s,\tilde
    r,\tilde\alpha,\varphi,\lambda)/\tilde h}
\right) a\,d\eta, 
\end{split} \]
where
\[  \tilde\Phi  =\tilde r(\eta\sin\varphi + \sqrt{\lambda g(0)-\eta^2}\cos\varphi)
+ \tilde\alpha\arcsin(\eta/\sqrt{\lambda g(0)})+{\mathcal O}(s) \,,
\]
and
\[  a = a ( s \tilde r \cos \varphi, s \tilde r \sin \varphi , \eta )
\,, \ \ a \in \CIc ( \mathbb R^3 ) \,. \]
For $s$ small enough, $\tilde\Phi$ satisfies the condition of
Proposition~\ref{l:3rd-derivative-estimate}. Indeed, under the change
of variables $\eta=\sqrt{\lambda g(0)}\sin\theta$ we get
$\tilde\Phi=\tilde r\sqrt{\lambda
g(0)}\cos(\theta-\varphi)+\tilde\alpha\theta+{\mathcal O}(s)$.
Therefore,
$$
|v|={\mathcal O}(s^{-1/2}\tilde h^{-1/6})={\mathcal O}(s^{-1/3}h^{-1/6}).
$$
This proves~\eqref{e:pole-est2-1} (note that the bound ${\mathcal
O}(h^{-1/2})$ follows directly from the integral representation).  To
show~\eqref{e:pole-est2-2}, note that for $|\alpha|<\varepsilon r$,
$\tilde\alpha$ is close to zero; then $\tilde\Phi$ is a Morse function
and we can use stationary phase in place of
Proposition~\ref{l:3rd-derivative-estimate} (see also the beginning of
the proof of that proposition).
\end{proof}

\noindent\textbf{Acknowledgements.} 
NB acknowledges partial support from Agence Nationale de
la Recherche project ANR-07-BLAN-0250. 
SD and MZ acknowledge partial
support by the National Science Foundation under the grant
DMS-0654436. They are also grateful to Universit\'e Paris-Nord
for its generous hospitality in the Spring 2011 when this paper
was written.


\end{document}